\documentclass[11pt]{amsart}
\usepackage{latexsym,amssymb,amsmath}
\usepackage{amsmath,amsfonts}
\usepackage{epsfig}
\usepackage{graphicx}
\usepackage{verbatim,xcolor}

\newcommand{\norm}[1]{\left\lVert#1\right\rVert}
\theoremstyle{plain}
\newtheorem{theorem}{Theorem}[section]

\newtheorem{Corollary}[theorem]{Corollary}
\newtheorem{lemma}[theorem]{Lemma}
\newtheorem{Definition}[theorem]{Definition}
\newtheorem{proposition}[theorem]{Proposition}

\theoremstyle{remark}
\newtheorem{Remark}[theorem]{Remark}

\allowdisplaybreaks

\newcommand{\Hmm}[1]{\leavevmode{\marginpar{\tiny%
$\hbox to 0mm{\hspace*{-0.5mm}$\leftarrow$\hss}%
\vcenter{\vrule depth 0.1mm height 0.1mm width \the\marginparwidth}%
\hbox to 0mm{\hss$\rightarrow$\hspace*{-0.5mm}}$\\\relax\raggedright #1}}}

\begin{document}

\title[Dynamics of the Hirota-Satsuma System]{SMOOTHING AND GLOBAL ATTRACTORS FOR THE HIROTA-SATSUMA SYSTEM ON THE TORUS}

\author[Ba\c{s}ako\u{g}lu, G\"{u}rel] {E. Ba\c{s}ako\u{g}lu, T. B. G\"{u}rel}
\thanks{The authors were partially supported by the T\"UB\.ITAK grant 118F152 and the Bo\u gazi\c ci University Research Fund grant BAP-14081.} 
\address{Department of Mathematics,Bo\u gazi\c ci University, 
Bebek 34342, Istanbul, Turkey}
\email{engin.basakoglu@boun.edu.tr}
\address{Department of Mathematics,
Bo\u gazi\c ci University, 
Bebek 34342, Istanbul, Turkey}
\email{bgurel@boun.edu.tr}

\begin{abstract}
We consider the Hirota-Satsuma system, a coupled KdV-type system, with periodic boundary conditions. The first part of the paper concerns with the smoothing estimates for the system. More precisely, it is shown that, for initial data in a Sobolev space, the difference of the nonlinear and linear evolutions lies in a smoother space. The smoothing gain we obtain depends very much on the arithmetic nature of the coupling parameter $a$ which determines the structure of the resonant sets in the estimates. In the second part, we address the forced and damped Hirota-Satsuma system and obtain counterpart smoothing estimates. As a consequence of these estimates, we prove the existence and smoothness of a global attractor in the energy space.

\end{abstract}

\maketitle
\section{Introduction}
The Hirota-Satsuma system is a system of coupled KdV equations, introduced by Hirota and Satsuma in $1981$, \cite{hirota}. In this paper, we consider the Hirota-Satsuma system with periodic boundary conditions 
\begin{align}\label{hir}
\begin{cases}
u_t+  au_{xxx}+3a (u^2)_x+\beta(v^2)_x =0, \hspace{0.5cm}x\in \mathbb{T}\\ v_t+v_{xxx}+3uv_x=0, \\ (u,v)\rvert_{t = 0}=(u_0,v_0)\in \dot{H}^s(\mathbb{T})\times H^s(\mathbb{T}) 
\end{cases}
\end{align}
where $a\in(\frac{1}{4},1)$, $\beta \in \mathbb{R}$ and $\dot{H}^s(\mathbb{T})=\{f\in H^s(\mathbb{T}): \int_0^{2\pi}f(x)\,\text{d}x=0\}$. Here the choice of the parameter $a$ is related to the resonance equations coming from after applying the normal form transformation to the system \eqref{hir}. The system \eqref{hir} is a generalization of the KdV equation (when $v=0$) and describes the interplay of two long waves evolving with different dispersion relations. Note that the mean zero condition on $v$ cannot be applicable since the system \eqref{hir} does not preserve the mean value of $v$, and that the momentum conservation holds for $u$ only:
\begin{align*}
 \int u(x,t)\, \text{d}x= \int u_0(x)\, \text{d}x.
\end{align*}
The system \eqref{hir} also satisfies the following conservation laws \cite{hirota}:
\begin{equation}\label{conservationhir}
\begin{split}
&E_1(u,v)=\int u^2-\frac{2\beta}{3}v^2\,\text{d}x,\hspace{0.5cm} \\&E_2(u,v)=\int (1-a)u_x^2-2\beta v_x^2-2(1-a)u^3+2\beta uv^2\,\text{d}x.
\end{split}
\end{equation}
As a consequence of these conserved energies, it turns out that the energy space for the system is $H^1\times H^1$. No other conserved quantities seem to exist for \eqref{hir} that holds for any $a$ and $\beta$; nevertheless for $a=-\frac{1}{2}$, the system is known to be completely integrable, \cite{hirota, ramani}. Before discussing the literature of the coupled KdV type systems, it makes sense to review more recent well-posedness results of the KdV equation
\begin{align}\label{kdv}
\begin{cases}
u_t+u_{xxx}+uu_x=0 \\ u(x,0)=u_0(x) \in H^s(K)\,\, \text{for}\,\,K=\mathbb{R}\,\,\text{or}\,\,\mathbb{T}.
\end{cases}   
\end{align}
 Introducing Fourier restriction spaces Bourgain extended the previous local well-posedness results of the KdV equation to the $L^2$ level on $\mathbb{R}$ and $\mathbb{T}$, \cite{Bourgain2}. Later in \cite{kenig}, Kenig, Ponce and Vega proved the local well-posedness in $H^s(\mathbb{R})$ for $s>-\frac{3}{4}$ and in $H^s(\mathbb{T})$ for $s>-\frac{1}{2}$. The local well-posedness in $H^{-\frac{3}{4}}(\mathbb{R})$ was established by Christ-Colliander-Tao in \cite{christ}. In \cite{colliander}, Colliander-Keel-Staffilani-Takaoka-Tao obtained the local and global well-posedness in $H^{-\frac{1}{2}}(\mathbb{T})$. The global well-posedness for $\mathbb{R}$ at the endpoint $s=-\frac{3}{4}$ was proved by Guo \cite{guo}. By using the integrability properties of the KdV equation, Kappeler and Topalov established the local and global well-posedness in $H^{-1}(\mathbb{T})$, \cite{kappeler}. Later, in \cite{molinet1,molinet2}, Molinet showed that the KdV equation is ill-posed in $H^s(K)$ for $K=\mathbb{R}, \mathbb{T}$, $s<-1$. Finally, the global well-posedness in $H^{-1}(\mathbb{R})$ has recently been obtained by Killip and Visan \cite{killip}, and the study of the well-posedness of \eqref{kdv} has been brought to a satisfactory conclusion.
 
The well-posedness theory of the Hirota-Satsuma system began with the work of He \cite{he}, with the assumptions $0<a<1$ and $\beta<0$ on the coefficients, He obtained the existence and uniqueness of the global solutions in $L^{\infty}([0,T]; H^{3}(K)\times H^{3}(K))$ for $K=\mathbb{T},\mathbb{R}$. In the real case, Feng \cite{feng} improved this result to the range $s\geq 1$ by considering slightly general coupled KdV-KdV system. In particular, it was shown that for $a\neq 1$ and $\beta<0$, the system is locally well-posed in $H^{s}(\mathbb{R})\times H^{s}(\mathbb{R})$ for $s\geq1$. Also with the additional assumption $0<a<1$, the system was shown to be globally well-posed in $H^{s}(\mathbb{R})\times H^{s}(\mathbb{R})$ for $s\geq1$. Later, Alvarez and Carvajal pushed the local result down to $s>\frac{3}{4}$ for the real case. They also showed that the system with $a\neq 0$ is ill-posed in $H^{s}(\mathbb{R})\times H^{s'}(\mathbb{R})$ for $s\in[-1,-3/4)$ and $s'\in\mathbb{R}$. Recently, Yang and Zhang \cite{zhang1} have studied the well-posedness of the Cauchy problem for a class of coupled KdV-KdV (cKdV) systems in $H^s(\mathbb{R})\times H^s(\mathbb{R})$, those including Gear-Grimshaw system, Hirota-Satsuma system, the Majdo-Biello system etc. In particular, regarding the Hirota-Satsuma system, they have given critical index $s^*\in\{-\frac{3}{4},0,\frac{3}{4}\}$ depending on the numeric value of the coefficient $a$ for which the Hirota-Satsuma Cauchy problem is locally well-posed in $H^s(\mathbb{R})\times H^s(\mathbb{R})$ when $s>s^*$. As regards to the periodic case, Angulo \cite{angulo} showed that for $a\neq 0,\,1$ the Hirota-Satsuma system is locally well-posed in $\dot{H}^s(\mathbb{T})\times \dot{H}^s(\mathbb{T})$ for $s\geq 1$. With the additional assumptions $\beta<0$ and $a<1$, Angulo further obtained the global well-posedness in the same space within the same Sobolev index range that follows from the conservation of the energy. Finally, in \cite{zhang}, Yang and Zhang have recently obtained the well-posedness results of the cKdV systems on the periodic domain $\mathbb{T}$ as a follow up of their corresponding work \cite{zhang1}. Here we shall merely summarize the well-posedness results of \cite{zhang} concerning the Cauchy problem \eqref{hir}. The results depend on the arithmetic properties of the coefficients $a$ and $\beta$. When $a=1$ and $\beta=0$, the system \eqref{hir} is locally well-posed in $\dot{H}^s(\mathbb{T})\times H^s(\mathbb{T})$ for $s\geq \frac{1}{2}$. In the case $a\in(-\infty,\frac{1}{4})\setminus \{0\}$, as the resonace interactions are relatively easier to control, the local well-posedness is established in $\dot{H}^s(\mathbb{T})\times H^s(\mathbb{T})$ for $s\geq -\frac{1}{4}$; whereas in the remaining regime, $a\in[\frac{1}{4},\infty)\setminus \{1\}$, the resonances raise special difficulties in which case one needs to know how well a given number can be approximated by rational numbers (Diophantine approximation). The idea of controlling resonances via the Diophantine approximation was initially implemented by Oh \cite{oh} to the Majdo-Biello system on the torus to establish the well-posedness. Using this approach, Yang and Zhang proved the local well-posedness in $\dot{H}^s(\mathbb{T})\times H^s(\mathbb{T})$ for $s\geq \min\{1,s_a+\}$ with the mean zero assumption on the initial data $u_0$. Here $s_a$ is defined by means of a number theoretic parameter based on the arithmetic properties of $a$. On account of the conserved energies \eqref{conservationhir} for the system \eqref{hir}, when $\frac{1}{4}\leq a<1$ and $\beta<0$, the local well-posedness can be upgraded to global well-posedness for $s\geq1$. Also when $a\in(-\infty,\frac{1}{4})\setminus \{0\}$ and $\beta<0$, the direct application of the conservation of $E_1(u,v)$ and the  corresponding local result yield the global well-posedness for $s\geq 0$.

In the first part of the paper, we study the smoothing property of the Hirota-Satsuma system, in other words, we prove that the difference between the nonlinear evolution and the linear evolution lies in a more regular space than the inital data under consideration. The proof is based on the method of normal forms through differentiation by parts introduced by Babin-Ilyin-Titi \cite{titi} and the Bourgain's Fourier restricted norm method, \cite{Bourgain2}. The idea of using combination of these methods in proving nonlinear smoothing effect on a bounded domain was first used by Erdo\u{g}an and Tzirakis for the KdV equation \cite{erdogan} and the Zakharov system \cite{erdoganzak}. The result for the KdV equation is somewhat surprising since the KdV equation is known to have no smoothing estimate on the real line. Recently, Compaan \cite{compaanmaj} studied the smoothing properties of the Majdo-Biello system on the torus and proved the existence of global attractors in the sense of arguments in \cite{erdogan,erdoganzak}. Here we continue the program initiated by these papers. In our proof, via the normal form reduction, the derivatives in the nonlinearities can be eliminated, and in return for this, the orders of the nonlinearities increase (from quadratic to cubic) and many resonant terms come into play based on the arithmetic properties of the coupling parameter $a$. In order to control the new trilinear nonlinearities we rely on the $X^{s,b}$ estimates. 

In the second part, we concentrate on the description of long time dynamics of the forced and weakly damped system:
\begin{align}\label{disshirfirst}
\begin{cases}
u_t+  au_{xxx}+\gamma_1u+3a (u^2)_x+\beta(v^2)_x =f \\ v_t+v_{xxx}+\gamma_2v+3uv_x=g \\ (u,v)\rvert_{t = 0}=(u_0,v_0)\in \dot{H}^1(\mathbb{T})\times H^1(\mathbb{T}) 
\end{cases}
\end{align}
where the damping coefficients $\gamma_1,\gamma_2$ are positive, $\beta<0$, $f,\,g$ are time independent, and $f\in H^1$ is mean zero, $g\in H^1$. To simplify the calculations, we will take $\gamma_1=\gamma_2$; the general case follows from minor modifications in the computations. The smoothing estimates obtained in the first part will play an essential role in demonstrating that the system \eqref{disshirfirst} possesses a global attractor, also the result here answers the regularity of the attractor above the energy level. Using the method of \cite{compaanmaj,erdoganlong,erdoganzak}, roughly, the idea is to write the solution in terms of linear and nonlinear parts and then to implement the smoothing estimates to the nonlinear part.
\subsection{Notation}
We write $A\lesssim B$ to denote that there is some constant $C>0$ such that $A\leq CB$, and $A\gtrsim B$ is defined accordingly. Moreover, $A\approx B$ means that $A\lesssim B$ and $B\lesssim A$. We write $A\ll B$ to denote that $A\leq \frac{1}{C}B$ for sufficiently large constant $C$. We also write $A\simeq B$ to designate that there is some small $\epsilon>0$ such that $|A-B|\leq \epsilon$. In order to simplify the calculations, the notation  $\mathcal{O}(\delta)$ is to be used in place of a constant $C\delta$ where $C$ might be dependent on the coupling parameter $a$ yet not on the variables involved in the calculations. We write $A+$ to denote $A+\epsilon$ for arbitrary $\epsilon>0$; similarly $A-$ is used to denote $A-\epsilon$. 

The Fourier sequence of a $2\pi$ periodic $L^2$ function $u$ is defined as
\begin{align*}
    u_k=\frac{1}{2\pi}\int_0^{2\pi}u(x)e^{-ikx}\,\text{d}x, \hspace{0.5cm}k\in\mathbb{Z}.
\end{align*}
By introducing the notation $\langle \cdot \rangle=\sqrt{1+|\cdot|^2}$, we define the Sobolev norms of $u\in H^s(\mathbb{T})$ as follows
\begin{align*}
    \norm{u}_{H^s}=\norm{\langle k \rangle^s \widehat{u}(k)}_{\ell^2_k}.
\end{align*}
Note that for a mean zero function $u\in\dot{H}^s$, $\norm{u}_{H^s}\approx \norm{|k|^s\widehat{u}(k)}_{\ell_k^2}$. The Fourier restriction spaces $X^{s,b}_a$ and $X^{s,b}_1$ associated to the system \eqref{hir} are defined via the norms
\begin{align*}
    &\norm{u}_{X^{s,b}_a}=\norm{\langle k\rangle^s\langle \tau-ak^3\rangle^b\widehat{u}(k,\tau)}_{\ell_k^2L^2_{\tau}}\\&\norm{v}_{X^{s,b}_1}=\norm{\langle k\rangle^s\langle \tau-k^3\rangle^b\widehat{v}(k,\tau)}_{\ell_k^2L^2_{\tau}}.
\end{align*}
The restricted norms are also defined by
\begin{align*}
    &\norm{u}_{X^{s,b}_{a,\delta}}=\inf_{\tilde{u}=u\,\text{on}\,[-\delta,\delta]}\norm{\tilde{u}}_{X^{s,b}_{a}}\\&\norm{v}_{X^{s,b}_{1,\delta}}=\inf_{\tilde{v}=v\,\text{on}\,[-\delta,\delta]}\norm{\tilde{v}}_{X^{s,b}_{}}.
\end{align*}

\section{Statement of Results}\label{prelimhir}
\subsection{Preliminaries}
In order to state well-posedness results and hence to study the smoothing properties of the system \eqref{hir} on the torus, we require special parameters related to given numbers, called irrationality exponent, which might be used to learn how close given numbers can be approximated by rational numbers. Irrationality exponent of a real number arises in the study of Diophantine approximation theory. In our discussion, we utilize these quantities in controlling resonances.     
\begin{Definition}[\cite{becher}]
A number $r\in \mathbb{R}$ is said to be approximable with power $\mu$, if the inequality \begin{align*}
    0<\Big|r-\frac{n}{k}\Big|<\frac{1}{|k|^{\mu}}
\end{align*}
holds for infinitely many $(n,k)\in \mathbb{Z}\times\mathbb{Z}^*$, and \begin{align}\label{irrexp}
    \mu(r)=\sup\{\mu\in\mathbb{R}: r\, \text{is approximable with power}\, \mu \}
\end{align}
is called the irrationality exponent of $r$.
\end{Definition}
We now review some properties of the irrationality exponent. Irrationality exponent maps the set of real numbers onto the set $\{1\}\cup [2,\infty]$, see \cite{jarnik, jarnik1}. In particular, for $r\in\mathbb{Q}$, $\mu(r)=1$ whereas for irrational number $r$ we have $\mu(r)\geq 2$. By the Thue-Siegel-Roth theorem \cite{roth,siegel,thue}, for an irrational algebraic number $r$, $\mu(r)=2$, also by the Khintchine theorem \cite{khintchine}, for almost every $r\in \mathbb{R}$, $\mu(r)=2$. The local theory for the Cauchy problem \eqref{hir} is based on the critical index $s_r$ for $r\geq\frac{1}{4}$ defined by
\begin{align*}
s_r= \begin{cases}
1 \hspace{1.5cm}\text{if}\,\,\mu(\rho_r)=1\,\,\text{or}\,\,\mu(\rho_r)\geq 3\\ \frac{\mu(\rho_r)-1}{2}\hspace{0.55cm}\text{if}\,\,2\leq \mu(\rho_r)<3 
\end{cases}   
\end{align*}
where $\rho_r=\sqrt{12r-3}$. In connection with the well-posedness of the system \eqref{hir}, we state some of the results of \cite{zhang}. 
\begin{theorem}[\cite{zhang}]
For $a\in[\frac{1}{4},\infty)\setminus\{1\}$ and $s\geq \min\{1,s_a+\}$, the Hirota-Satsuma system \eqref{hir} is locally well-posed in $\dot{H}^s\times H^s$. In particular, given any initial data $(u_0,v_0)\in\dot{H}^s\times H^s $ there exists $\delta\approx (\norm{u_0}_{H^s}+\norm{v_0}_{H^s})^{-\frac{3}{2}}$ and a unique solution 
\begin{align*}
(u,v)\in C([-\delta,\delta];H^s\times H^s)\end{align*} satisfying $\norm{u}_{X^{s,1/2}_{a,\delta}}+\norm{v}_{X^{s,1/2}_{1,\delta}}\lesssim \norm{u_0}_{H^s}+\norm{v_0}_{H^s}.$
\end{theorem}
 
\begin{theorem}[\cite{zhang}]
Let $\beta<0$ and $a\in[\frac{1}{4},1)$. Then for $s\geq 1$, the Hirota-Satsuma initial value problem is globally well-posed in $\dot{H}^s\times H^s$.
\end{theorem}
\subsection{Smoothing Estimates for the Hirota-Satsuma System}
Applying normal form transformation to the system \eqref{hir} leads to the resonance equations (for $k_1+k_2=k$)
\begin{align*}
ak^3-k_1^3-k_2^3=-3k(k_1-r_1k)(k_1-r_2k)\,\,\text{and}\,\,k^3-ak_1^3-k_2^3=(1-a)k_1(k_1-\widetilde{r}_1k)(k_1-\widetilde{r}_2k)   
\end{align*}
where 
\begin{align}\label{constantshir}
r_1=\frac{1}{2}+\frac{\sqrt{12a-3}}{6},\,\,\, r_2=\frac{1}{2}-\frac{\sqrt{12a-3}}{6}\,\,\text{and}\,\,\,  \widetilde{r}_1=1/r_1,\,\,\,\widetilde{r}_2=1/r_2.
\end{align}
Note that $r_1$, $r_2$ are the roots of the equation $3x^2-3x+(1-a)$, which immediately implies that $\widetilde{r}_1$, $\widetilde{r}_2$ are the roots of the equation $(1-a)x^2-3x+3$. Therefore, $r_j$ and $\widetilde{r}_j$ are algebraic only when $a\in\mathbb{Q}$. By \eqref{constantshir}, we notice that $r_1, r_2\in\mathbb{R}$ if and only if $a\geq\frac{1}{4}$, and that $\widetilde{r}_1, \widetilde{r}_2\in\mathbb{R}$ if and only if $a\in [\frac{1}{4},1)\cup (1,\infty)$. In this paper, we consider the problem \eqref{hir} for $a\in(\frac{1}{4},1)$, also the problem for interval $(1,\infty)$ can be handled in a similar vein. As performing smoothing estimates, we will be dealing with many expressions such as $ak^3-k_1^3-k_2^3$ and $k^3-ak_1^3-k_2^3$ that appear in the denominators. Controlling such expressions in the case they get close to zero relies on the following lemma:   
\begin{lemma}[\cite{zhang}]
If $r\in \mathbb{R}\setminus \mathbb{Q}$ with $\mu(r)<\infty$, then for any $\epsilon>0$, there exists a constant $K=K(r,\epsilon)>0$ such that the inequality
\begin{align}\label{resonanceineq}
\Big|r-\frac{n}{k}\Big|\geq \frac{K}{|k|^{\mu(r)+\epsilon}}
\end{align} 
holds for any $(n,k)\in \mathbb{Z}\times \mathbb{Z}^*$.
\end{lemma}
Note that when $r_j$, $\widetilde{r}_j$, as introduced in \eqref{constantshir}, are rational numbers, by the Proposition \ref{base} below, we will still be able to use the inequality \eqref{resonanceineq} for these numbers with  $\mu(r_j)=\mu(\widetilde{r}_j)=1$. Next lemma collects some invariance properties of the irrationality exponent, which is proved in \cite{zhang}.
\begin{lemma} The irrationality exponent defined by \eqref{irrexp} satisfies
\begin{enumerate}
    \item For any $r\in\mathbb{R}$ and $q\in\mathbb{Q}$, $\mu(q+r)=\mu(r)$,
    \item For any $r\in\mathbb{R}$ and $q\in\mathbb{Q}\setminus\{0\}$, $\mu(qr)=\mu(r)$,
    \item For any $r\in\mathbb{R}\setminus\{0\}$, $\mu(\frac{1}{r})=\mu(r)$.
\end{enumerate}
\end{lemma}
Using the lemma above we may write
\begin{align}\label{muequalities}
\mu(\widetilde{r}_2)=\mu(\widetilde{r}_1)=\mu(r_1)=\mu(r_2)=\mu(\sqrt{12a-3})=:\mu(\rho_a).    
\end{align}
With the notations introduced above we state our smoothing result as follows: 
\begin{theorem}\label{smoothingresulthir}
Fix $s>\frac{1}{2}$ and $a\in(\frac{1}{4},1)$. Consider the solution of \eqref{hir} with initial data $(u(0,x),v(0,x))=(u_0(x),v_0(x))\in \dot{H}^s(\mathbb{T})\times H^s(\mathbb{T})$. Then for any $s_1-s<\min\{1,s-\frac{1}{2}, s+2-\mu(\rho_a), 2s+1-\mu(\rho_a)\}$,
we have \begin{align}
   \label{smoothinghir1} &u(t,x)-e^{-ta \partial_x^3}u_0\in C_t^0H^{s_1}_x, \\& \label{smoothinghir2} v(t,x)-e^{-t\partial_x^3}v_0\in C_t^0H^{s_1}_x.
\end{align}
For almost every $a$, \eqref{smoothinghir1} and \eqref{smoothinghir2} hold with $s_1-s<\min\{1,s-\frac{1}{2}\}$. When $a=\frac{3p(p-q)+q^2}{q^2}$, $(p,q)\in \mathbb{Z}\times\mathbb{Z}$ with $\frac{q}{2}<p<q$, fix $s>1$, then the smoothing statements \eqref{smoothinghir1} and \eqref{smoothinghir2} are valid for $s_1-s\leq\min\{1-,s-1\}$ instead.
Assume that we have a growth bound  $\norm{u(t)}_{H^s}+\norm{v(t)}_{H^s}\lesssim \langle t\rangle^{\beta(s)}$,
for some $\beta(s)>0$. Then 
\begin{align*}
    \norm{u(t)-e^{-ta \partial_x^3}u_0}_{H^{s_1}}+\norm{v(t)-
e^{-t \partial_x^3}v_0}_{H^{s_1}}\leq C(a,s,s_1,\norm{u_0}_{H^s},\norm{u_0}_{H^s})\langle t\rangle^{1+\frac{9}{2}\beta(s)}.
\end{align*}
\end{theorem}
\begin{Remark}
In the case the coefficients $r_1$ and $r_2$ in \eqref{constantshir} are rational numbers, which is the case when $a$ is a rational of the form $\frac{3p(p-q)+q^2}{q^2}$ as stated in the above theorem, we have to control some additional terms due to the resonances, in which case smoothing is attained solely for $s>1$. On the other side, if the coefficients $r_1$, $r_2$ are irrational algebraic numbers, which is the case when $a$ is a rational number such that $a\neq\frac{3p(p-q)+q^2}{q^2}$ with $\frac{q}{2}<p<q$, then by \eqref{muequalities}, $\mu(\rho_a)=2$ in Theorem \ref{smoothingresulthir} yields the best possible smoothing. Indeed, the best regularity gain given by the Theorem \ref{smoothingresulthir} is reached for almost every $a\in(\frac{1}{4},1)$. As a consequence, with regards to smoothing, the above discussion shows how the regularity level is unstable under a slight perturbation of $a$ within $(\frac{1}{4},1)$.    
\end{Remark}
Using smoothing estimates one can obtain growth bounds for higher order Sobolev norms in the lack of complete integrability. As a corollory of the smoothing theorem above, we obtain the following result.

\begin{Corollary}
For any $s\geq1$ and almost every $a\in(\frac{1}{4},1)$ for which $\mu(\rho_a)=2$, the global solution of \eqref{hir} with $\dot{H}^s\times H^s$ data satisfies the growth bound 
\begin{align*}
\norm{u(t)}_{H^s}+\norm{v(t)}_{H^s}\leq C_0(1+|t|)^{C_1}    
\end{align*}
where $C_0$ depends on $a,s,\norm{u_0}_{H^s},\norm{v_0}_{H^s}$ and $C_1$ depends on $s$.
\end{Corollary}
\begin{proof}
Due to the conserved energies, $H^1$ norms of $u$ and $v$ are bounded for all times. The idea is to use the result of Theorem \ref{smoothingresulthir} repeatedly to obtain the growth bound for the Sobolev norms above the energy level. To use this, suppose that the claim of the corollory holds for $s\geq1$. Hence for $s_1\in(s,s+\min\{1,s-\frac{1}{2}\})$, Theorem \ref{smoothingresulthir} leads to
\begin{align*}
\norm{u(t)-e^{-at\partial_x^3}u_0}_{H^{s_1}}+\norm{v(t)-e^{-t\partial_x^3}v_0}_{H^{s_1}}\leq C(1+|t|)^{1+\frac{9}{2}\beta(s)},   
\end{align*}
from this and the fact that linear groups are unitary, we get
\begin{align*}
\norm{u(t)}_{H^{s_1}}+\norm{v(t)}_{H^{s_1}}\leq \norm{u_0}_{H^{s_1}}+\norm{v_0}_{H^{s_1}}+C_0(1+|t|)^{C_1}.
\end{align*}
Continuing iteratively in this way, we reach any index $s_1$.
\end{proof}

\subsection{Existence of a Global Attractor for the Hirota-Satsuma System}
The essential tool in establishing the existence of a global atractor for the dissipative Hirota-Satsuma system \eqref{disshirfirst} in this context will be the smoothing estimates which are to be discussed later. The existence of global attractors makes sense only for the globally well-posed problems, in this regard we note that the initial value problem \eqref{disshirfirst} is locally and globally well-posed in $\dot{H}^1\times H^1$. The local well-posedness follows by using the estimates of \cite{zhang} and the global well-posedness follows from the energy estimate of Lemma \ref{energyhir} which also implies the existence of an absorbing ball. To set the stage for the description of the problem, let $U(t)$ be the semigroup operator mapping data to solution in the phase space.  
\begin{Definition}[\cite{Temam}]
We say that a compact subset $\mathcal{A}$ of the phase space $H$ is a global attractor for the semigroup $\{U(t)\}_{t\geq 0}$, if $U(t)\mathcal{A}=\mathcal{A}$ for all $t>0$, and $$\lim\limits_{t\rightarrow \infty}d(U(t)g,\mathcal{A})=0,\,\,\forall g\in H.$$
\end{Definition}
Next we define an absorbing set into which all solutions enter eventually.  
\begin{Definition}[\cite{Temam}]
A bounded subset $\mathcal{B}$ of a phase space $H$ is called an absorbing set if for any bounded $\mathcal{S}\subset H$, there exists $T=T(\mathcal{S})$ such that $U(t)\mathcal{S}\subset \mathcal{B}$ for all $t\geq T$.
\end{Definition}
Note that the existence of a global attractor for a semigroup implies the existence of an absorbing set whereas the converse holds only when a semigroup is asymptotically compact:
\begin{theorem}[\cite{Temam}]\label{attractortemam}
Suppose that $H$ is a metric space and $U(t):H\rightarrow H$ is a continuous semigroup defined for all $t\geq 0$. Furthermore suppose that there is an absorbing set $\mathcal{B}$. If the semigroup $\{U(t)\}_{t\geq 0}$ is asymptotically compact, that is, for every bounded sequence $\{x_k\}$ in $H$ and every sequence of times $t_k\rightarrow \infty$, the set $\{U(t_k)x_k\}_k$ is relatively compact in $H$, then the $\omega$-limit set $$\omega(\mathcal{B})=\bigcap\limits_{s\geq0}\overline{\bigcup\limits_{t\geq s} U(t)\mathcal{B}}$$
is a global attractor, where the closure is taken on $H$.
\end{theorem}
Therefore, in the presence of an absorbing set the proof of existence of a global attractor reduces to proving asymptotic compactness of the evolution operator. In our discussion, the proof of the asymptotic compactness of the flow depends very much on the smoothing estimate for the forced and weakly damped system that we establish later. Thus we have the following result:
\begin{theorem}\label{attrhir}
Consider the forced and weakly damped Hirota-Satsuma system \eqref{disshirfirst} on $\mathbb{T}\times[0,\infty)$ with $(u_0,v_0)\in \dot{H}^1\times H^1$. Then, for almost every $a\in(\frac{1}{4},1)$, the equation has a global attractor in $\dot{H}^1\times H^1$. Moreover, for any $\alpha<\frac{1}{2}$, the global attractor is a compact subset of $H^{1+\alpha}\times H^{1+\alpha}$. 
\end{theorem}

\section{Proof of Theorem \ref{smoothingresulthir} }
We start by writing \eqref{hir} in an equivalent form through differentiation parts. Below $\sum^*$ denotes the summation over all terms for which the corresponding denominator is never zero. The requirement of the nonzero denominators for the sums not written by this notation is provided by the mean zero assumption on $u_0$. 
\begin{proposition} \label{base} 
 Let $u_0\in\dot{H}^s$. The system \eqref{hir} can be written as 
\begin{multline*}
\begin{cases} 
\partial_t\big(e^{-ia k^3t}[u_k+B_1(u,u)_k+B_2(v,v)_k]\big)=e^{-ia k^3t}\big[R_1(u,v,v)_k+R_2(u,u,u)_k\\\hspace{8.9cm}+R_3(u,v,v)_k+\rho_1(u,u)_k+\rho_2(v,v)_k\big] \\\partial_t\big(e^{-ik^3t}[v_k+B_3(u,v)_k]\big)=e^{-ik^3t}\big[R_4(u,u,v)_k+\frac{\beta}{3a}R_4(v,v,v)_k+R_5(u,u,v)_k+\rho_3(u,v)_k\big] 
\end{cases} 
\end{multline*} 
 where the terms are defined as follows
\begin{align*}
 &B_1(f,g)_k=-\sum_{k_1+k_2=k}\frac{f_{k_1}g_{k_2}}{k_1k_2},\,\,B_2(f,g)_k=-\beta k\sum_{k_1+k_2=k}^*\frac{f_{k_1}g_{k_2}}{a k^3-k_1^3-k_2^3}\\& B_3(f,g)_k=-3 \sum_{k_1+k_2=k}^* \frac{k_2 f_{k_1}g_{k_2}}{k^3-a k_1^3-k_2^3}   
 \\&R_1(f,g,h)_k=-2i\beta \sum_{k_1+k_2+k_3=k}^* \frac{k_2 f_{k_1}g_{k_2}h_{k_3}}{(k_1+k_2-r_1k)(k_1+k_2-r_2k)}\\& R_2(f,g,h)_k=6ia \underset{\substack{k_1+k_2+k_3=k\\ (k_1+k_2)(k_2+k_3)(k_3+k_1)\neq 0}}\sum \frac{f_{k_1}g_{k_2}h_{k_3}}{k_1}\\&
 R_4(f,g,h)_k=9ia\sum_{k_1+k_2+k_3=k}^* \frac{k_3(k_1+k_2)f_{k_1}g_{k_2}h_{k_3}}{k^3-a (k_1+k_2)^3-k_3^3}\\& R_5(f,g,h)_k=9i\sum_{k_1+k_2+k_3=k}^* \frac{k_3(k_2+k_3) f_{k_1}g_{k_2}h_{k_3}}{k^3-a k_1^3-(k_2+k_3)^3}\\& \rho_1(f,g)_k=-\frac{6ia}{k}|f_k|^2g_k, \hspace{1cm} \rho_2(f,g)_k=-2i\beta kf_{r_1k}g_{r_2k}\\& \rho_3(f,g)_k= -3ik\big[(1-\tilde{r}_1)f_{\tilde{r}_1k}g_{(1-\tilde{r}_1)k}+(1-\tilde{r}_2)f_{\tilde{r}_2k}g_{(1-\tilde{r}_2)k}\big].
\end{align*}
\end{proposition}

\begin{proof}
Writing \eqref{hir} on the Fourier side we have
\begin{align*}
\begin{cases}
\partial_t u_k-ia k^3u_k+3ia k  \sum\limits_{k_1+k_2=k}u_{k_1}u_{k_2}  +i\beta k\sum\limits_{k_1+k_2=k}v_{k_1}v_{k_2}=0 \\ \partial_t v_k-ik^3v_k+3i \sum\limits_{k_1+k_2=k}k_2u_{k_1}v_{k_2}=0.
\end{cases}
\end{align*}
Via the substitution $f_k(t)=e^{-ia k^3t}u_k(t)$ and $g_k(t)=e^{-ik^3t}v_k(t)$, the above system transforms to
\begin{align}\label{reducedeq}
\begin{cases}
\partial_tf_k=-3iak \sum\limits_{k_1+k_2=k}e^{-ia t(k^3-k_1^3-k_2^3)}f_{k_1}f_{k_2}-i\beta k \sum\limits_{k_1+k_2=k}e^{-it(a k^3-k_1^3-k_2^3)}g_{k_1}g_{k_2} \\ \partial_tg_k=-3i\sum\limits_{k_1+k_2=k}k_2e^{-it(k^3-a k_1^3-k_2^3)}f_{k_1}g_{k_2}.
\end{cases}
\end{align}
Implementing differentiation by parts to the first equation in \eqref{reducedeq} we get 
\begin{multline*}
\partial_tf_k=\sum\limits_{k_1+k_2=k}\frac{\partial_t(e^{-3i kk_1k_2t}f_{k_1}f_{k_2})}{k_1k_2}-\sum\limits_{k_1+k_2=k}\frac{e^{-3ia kk_1k_2t}\partial_t(f_{k_1}f_{k_2})}{k_1k_2}\\+\beta k \sum\limits_{k_1+k_2=k}^*\frac{\partial_t(e^{-it(a k^3-k_1^3-k_2^3)}g_{k_1}g_{k_2})}{a k^3-k_1^3-k_2^3}-\beta k\sum\limits_{k_1+k_2=k}^*\frac{e^{-it(a k^3-k_1^3-k_2^3)}\partial_t(g_{k_1}g_{k_2})}{a k^3-k_1^3-k_2^3}-2i\beta kg_{r_1k}g_{r_2k}.
\end{multline*}
Note here that $a k^3-k_1^3-k_2^3=-3k(k_1-r_1k)(k_1-r_2k)$ where $r_1$ and $r_2$ are the roots of the quadratic equation $3x^2-3x+(1-a)$. So the resonant term corresponding to the second sum of the first equation in \eqref{reducedeq} come up when $r_1k\in \mathbb{Z}$, in which case we would have $r_1,r_2\in \mathbb{Q}$. There is no contribution from $k=0$ solution to the resonant term owing to the mean zero assumption on $u_0$. Following the arguments of \cite{titi}, \cite{erdogan} and using the first line of \eqref{reducedeq}, we have
\begin{multline*}
\sum\limits_{k_1+k_2=k}\frac{e^{-3ia kk_1k_2t}\partial_t(f_{k_1}f_{k_2})}{k_1k_2}=-6ia\underset{\substack{k_1+k_2+k_3=k \\ (k_1+k_2)(k_2+k_3)(k_3+k_1)\neq 0}}\sum \frac{e^{-3ia(k_1+k_2)(k_2+k_3)(k_3+k_1)t}}{k_1}f_{k_1}f_{k_2}f_{k_3}\\-2i\beta \underset{\substack{k_1+k_2+k_3=k}}\sum \frac{e^{-it(a k^3-a k_1^3-k_2^3-k_3^3)}}{k_1}f_{k_1}g_{k_2}g_{k_3}+\frac{6ia}{k}|f_k|^2f_k.
\end{multline*}
Likewise using the second line of \eqref{reducedeq}, the fourth sum in $\partial_tf_k$ can be written as
\begin{align*}
-2i\beta \sum\limits_{k_1+k_2+k_3=k}^* \frac{k_2e^{-it(a k^3-a k_1^3-k_2^3-k_3^3)}}{(k_1+k_2-r_1k)(k_1+k_2-r_2k)}f_{k_1}g_{k_2}g_{k_3}.
\end{align*}
As for the second equation in \eqref{reducedeq}, again we use differentiation by parts to get 
\begin{multline*}
\partial_tg_k=3\sum\limits_{k_1+k_2=k}^*\frac{k_2\partial_t(e^{-it(k^3-a k_1^3-k_2^3)}f_{k_1}g_{k_2})}{k^3-a k_1^3-k_2^3}- 3\sum\limits_{k_1+k_2=k}^*\frac{k_2e^{-it(k^3-a k_1^3-k_2^3)}}{k^3-a k_1^3-k_2^3}(\partial_t f_{k_1}g_{k_2})\\-3ik\big[(1-\tilde{r}_1)f_{\tilde{r}_1k}g_{(1-\tilde{r}_1)k}+(1-\tilde{r}_2)f_{\tilde{r}_2k}g_{(1-\tilde{r}_2)k}\big].
\end{multline*}
Note that in obtaining the resonant term we use the identity $k^3-a k_1^3-k_2^3=(1-a)k_1(k_1-\tilde{r}_1k)(k_1-\tilde{r}_2k)$ where $\tilde{r}_j=1/r_j$, $j=1,2$. By the mean zero assumption on $u_0$, the only contribution comes just when $\tilde{r}_1k,\ \tilde{r}_2k \in \mathbb{Z}$ in which case we need to have $\tilde{r}_1, \tilde{r}_2 \in \mathbb{Q}$. Using \eqref{reducedeq}, we rewrite the second sum above as follows
\begin{multline*}
3\sum\limits_{k_1+k_2=k}^*\frac{e^{-it(k^3-a k_1^3-k_2^3)}k_2}{k^3-a k_1^3-k_2^3}(\partial_t f_{k_1}g_{k_2})=-9ia \sum\limits_{k_1+k_2+k_3=k}^*\frac{e^{-it(k^3-a k_1^3-a k_2^3-k_3^3)}(k_1+k_2)k_3}{k^3-a (k_1+k_2)^3-k_3^3}f_{k_1}f_{k_2}g_{k_3}\\-3i\beta \sum\limits_{k_1+k_2+k_3=k}^*\frac{e^{-it(k^3-k_1^3- k_2^3-k_3^3)}(k_1+k_2)k_3}{k^3-a (k_1+k_2)^3-k_3^3}g_{k_1}g_{k_2}g_{k_3}\\-9i \sum\limits_{k_1+k_2+k_3=k}^*\frac{e^{-it(k^3-a k_1^3-a k_2^3-k_3^3)}(k_2+k_3)k_3}{k^3-a k_1^3-(k_2+k_3)^3}f_{k_1}f_{k_2}g_{k_3}.
\end{multline*} Bringing all the terms together and reinstating the $u$ and $v$ variables yield the assertion.     
\end{proof}
 We integrate the system in Proposition \ref{base} from $0$ to $t$ to get
 \begin{multline}\label{estimatefourierhir} 
 \begin{cases}
 u_k(t)-e^{ia k^3t}u_k(0)=-B_1(u,u)_k(t)-B_2(v,v)_k(t)+e^{ia k^3t}\big[B_1(u,u)_k(0)+B_2(v,v)_k(0)\big]\\ \hspace{4cm}+\displaystyle\int_0^te^{ia k^3(t-s)}\big[R_1(u,v,v)_k(s)+R_2(u,u,u)_k(s)+R_3(u,v,v)_k(s)\\ \hspace{10cm}+\rho_1(u,u)_k(s)+\rho_2(v,v)_k(s)\big]ds \\ v_k(t)-e^{ik^3t}v_k(0)=-B_3(u,v)_k(t)+e^{ik^3t}B_3(u,v)_k(0) + \displaystyle\int_0^te^{ik^3(t-s)}\big[R_4(u,u,v)_k(s)\\ \hspace{6cm} +\frac{\beta}{3a}R_4(v,v,v)_k(s)+R_5(u,u,v)_k(s)+\rho_3(u,v)_k(s)\big]ds
\end{cases}
\end{multline}
In the following we give the proofs of a priori estimates for $\rho_j$ and $B_j$, $j=1,2,3$.

\begin{proposition}\label{hirprop01}
 For $s_1-s\leq 2s+1$, we have $$\norm{\rho_1(u,v)}_{H^{s_1}_x}\lesssim \norm{u}_{H^s_x}^2\norm{v}_{H^s_x}.$$
 For $s_1-s\leq s-1$,
 \begin{align*}
   \norm{\rho_j(u,v)}_{H^{s_1}_x}\lesssim \norm{u}_{H^s_x}\norm{v}_{H^s_x},\,\,j=2,3.  
 \end{align*}
 Also for $s_1-s\leq 1$,
 \begin{align*}
   \norm{B_1(u,v)}_{H^{s_1}_x}\lesssim \norm{u}_{H^s_x}\norm{v}_{H^s_x}. 
 \end{align*}
\end{proposition}
\begin{proof}
The proofs of the first and third inequalities were given in \cite{erdogan}. As for the other terms, for $s_1-s\leq s-1$, we have
\begin{multline*}
\norm{\rho_2(u,v)}_{H^{s_1}_x}\lesssim \norm{\langle k\rangle^{s_1-2s+1} \langle r_1k\rangle^su_{r_1k} \langle r_2k\rangle^sv_{r_2k}}_{\ell_k^2}\\\lesssim \norm{\langle k\rangle^su_k}_{\ell_k^{\infty}}\norm{\langle k\rangle^sv_k}_{\ell_k^2}\lesssim \norm{u}_{H^s_x}\norm{v}_{H^s_x}, 
\end{multline*}
where the last inequality is due to $\ell^2\hookrightarrow \ell^{\infty}$. $\rho_3$ estimate follows by the same argument as well. \end{proof}

\begin{proposition}\label{hirprop02}
 For $s>\frac{1}{2}$ and $s_1-s<\min\{1,s+2-\mu \}$, we have $$\norm{B_2(u,v)}_{H^{s_1}_x}\lesssim \norm{u}_{H^s_x}\norm{v}_{H^s_x}.$$
  When $a=\frac{3p(p-q)+q^2}{q^2}$, $(p,q)\in \mathbb{Z}\times\mathbb{Z}$ with $\frac{q}{2}<p<q$, we replace the above requirement by $s_1-s\leq 1$.
 \end{proposition}
 \begin{proof}
 It is enough solely to consider the case $|k_1|\gtrsim |k_2|$ by the symmetry. 
  \\{\bf Case A.} $|k_1-r_1k|\geq\frac{1}{3}$, $|k_1-r_2k|\geq \frac{1}{3}$\\
 In this region,
 $$|ak^3-k_1^3-k_2^3|=|3k(k_1-r_1k)(k_1-r_2k)|\geq |k|\max\{|k_1-r_1k|,|k_1-r_2k\}\gtrsim (r_1-r_2)|k|^2.$$
 Accordingly, since $|k_1|\gtrsim |k|$, we have the  estimate \begin{align*}
  \norm{B_2(u,v)}_{H^{s_1}_x}\lesssim \norm{\sum_{k_1+k_2=k}\langle k \rangle^{s_1-1}|u_{k_1}||v_{k_2}|} _{l^2_k}\lesssim \norm{\langle k \rangle^{s_1-s-1}\Big(|u_k|\langle k\rangle^s*|v_k|\frac{\langle k \rangle^s}{\langle k \rangle^s}\Big)}_{l^2_k}  
 \end{align*}
 Assuming that $s_1-s\leq 1$ and using Young's and H\"{o}lder's inequalities successively, the last norm above is majorized by
 \begin{align*}
     \norm{\langle k\rangle^su_k}_{l^2_k}\norm{\langle k\rangle^sv_k\langle k\rangle^{-s}}_{l^1_k}\lesssim \norm{u}_{H^s}\norm{\langle k\rangle^su_k}_{l^2_k}\norm{\langle k\rangle^{-s}}_{l^2_k}\lesssim \norm{u}_{H^s}\norm{v}_{H^s}.
 \end{align*}
 {\bf Case B.} $|k_1-r_1k|<\frac{1}{3}$ or $|k_1-r_2k|< \frac{1}{3}$\\
 We consider the case $|k_1-r_1k|<\frac{1}{3}$ only, the other one is similar. Then as $r_1+r_2=1$, $k_1\simeq r_1k$ and $k_2\simeq r_2k$. This means that the sum under consideration consists of a single term of order $\approx k$. We next make use of the estimate due to irrationality exponent of $r_1$:
 $$|k_1-r_1k|=|k|\Big|r_1-\frac{k_1}{k}\Big|\geq |k|\frac{K(r_1,\epsilon)}{|k|^{\mu(r_1)+\epsilon}}$$
 for any $\epsilon>0$. This allows us to estimate the multiplier in the definition of $B_2$ as follows:
 \begin{multline*}
   |ak^3-k_1^3-k_2^3|=3|k||k_1-r_1k||k_1-r_2k|\geq 3K(r_1,\epsilon)|k|^{2-\mu(r_1)-\epsilon}\big[(r_1-r_2)|k|-1/3\big]\\ \gtrsim |k|^{3-\mu(r_1)-\epsilon}  
 \end{multline*}
 Therefore, using H\"{o}lder's inequality and the embedding $\ell^2\hookrightarrow \ell^4$, we obtain that
 \begin{align*}
 \norm{B_2(u,v)}_{H^{s_1}_x}&\lesssim \norm{\langle k \rangle^{(s_1+\mu(r_1)-2+\epsilon)/2}u_k}_{\ell_k^4}\norm{\langle k \rangle^{(s_1+\mu(r_1)-2+\epsilon)/2}v_k}_{\ell_k^4}\\&\lesssim \norm{\langle k \rangle^{(s_1-2s+\mu(r_1)-2+\epsilon)/2}\langle k \rangle^{s}u_k}_{\ell_k^2}\norm{\langle k \rangle^{(s_1-2s+\mu(r_1)-2+\epsilon)/2}\langle k \rangle^{s}v_k}_{\ell_k^2}\\&\lesssim \norm{u}_{H^s_x}\norm{v}_{H^s_x}  
 \end{align*}
where the last inequality stems from the assumption that $s_1-s<s+2-\mu(r_1)$. 
\end{proof}

\begin{proposition}\label{hirprop03}
 Assume that $u\in \dot{H}^s$. For $s>\frac{1}{2}$ and $s_1-s<\min\{1,s+2-\mu \}$, we have $$\norm{B_3(u,v)}_{H^{s_1}_x}\lesssim \norm{u}_{H^s_x}\norm{v}_{H^s_x}.$$
 When $a=\frac{3p(p-q)+q^2}{q^2}$, $(p,q)\in \mathbb{Z}\times\mathbb{Z}$ with $\frac{q}{2}<p<q$, we replace the above requirement by $s_1-s\leq 1$.
 \end{proposition}
\begin{proof}
Since there is no symmetry in this case, we consider the two cases:
\\{\bf Case A.} $|k_1|\geq |k|$\\
In this region, $|k_1|\gtrsim |k_2|$, so it suffices to show that
\begin{align}\label{1+2}
\norm{\sum^*_{k_1+k_2=k}\frac{\langle k \rangle^{s_1}u_{k_1}v_{k_2}}{(k_1-\widetilde{r}_1k)(k_1-\widetilde{r}_2k)}}_{\ell_k^2} \lesssim \norm{u}_{H^s_x}\norm{v}_{H^s_x}.  
\end{align}
{\bf Case A.1.} $|k_1-\widetilde{r}_1k|\geq \delta$, $|k_1-\widetilde{r}_2k|\geq \delta$\\
Notice that $$|(k_1-\widetilde{r}_1k)(k_1-\widetilde{r}_2k)|\geq \delta\max\{|k_1-\widetilde{r}_1k|,|k_1-\widetilde{r}_2k|\}\geq \delta(\widetilde{r}_2-\widetilde{r}_1)|k|\gtrsim |k|$$ which yields the estimate:
\begin{align*}
 \text{LHS of}\,\eqref{1+2}\lesssim \norm{\sum_{k_1+k_2=k}\langle k\rangle^{s_1-1}|u_{k_1}||v_{k_2}|} _{l^2_k}   
\end{align*}
but this has already been handled in the proof of the previous proposition, thus the estimate in \eqref{1+2} holds when $s_1-s\leq 1$.
\\{\bf Case A.2.} $|k_1-\widetilde{r}_1k|< \delta$ or $|k_1-\widetilde{r}_2k|< \delta$\\
We consider the first case $|k_1-\widetilde{r}_1k|< \delta$. Second case is analogous. In this case, we have $k_1\simeq \widetilde{r}_1k$ and $k_2=k-k_1\simeq(1-\widetilde{r}_1)k$, and hence $|k|\approx|k_1|\approx|k_2|$. As a result, the values of $k_1$ and $k_2$ in the sum are dependent on $k$. Therefore, the bound
$$|(k_1-\widetilde{r}_1k)(k_1-\widetilde{r}_2k)|\gtrsim |k|^{1-\mu(\widetilde{r}_1)+\epsilon}|k_1-\widetilde{r}_2k|\gtrsim |k|^{2-\mu(\widetilde{r}_1)-\epsilon}$$ implies that
\begin{multline*}
\text{LHS of}\,\eqref{1+2}\lesssim \norm{\langle k \rangle^{s_1-2+\mu(\widetilde{r}_1)+\epsilon}u_{k_1}v_{k_2}}_{\ell_k^2}\\\lesssim \norm{\langle k \rangle^{(s_1-2+\mu(\widetilde{r}_1)+\epsilon)/2}u_{k}}_{\ell_k^4}\norm{\langle k \rangle^{(s_1-2+\mu(\widetilde{r}_1)+\epsilon)/2}v_{k}}_{\ell_k^4}\lesssim \norm{u}_{H^s_x}\norm{v}_{H^s_x}
\end{multline*}
provided that $s_1-s<s+2-\mu(\widetilde{r}_1)$.
\\{\bf Case B.} $|k_1|<|k|$\\
In this region $|k_2|\lesssim |k|$. Thus \begin{align} \label{B3}
  \norm{B_3(u,v)}_{H^{s_1}_x}\lesssim \norm{\sum^*_{k_1+k_2=k}\frac{\langle k \rangle^{s_1+1} u_{k_1}v_{k_2}}{k^3-ak_1^3-k_2^3}} _{l^2_k}.  
 \end{align}
 By the mean zero presumption on $u$, $k_1\neq 0$, thus we may write $k_1=\eta k$ for some $|k|^{-1}\leq |\eta|<1$. It follows that \begin{align*}
|k^3-ak_1^3-k_2^3|=|\eta k^3||(1-a)\eta^2+3-3\eta|\geq |k|^2 |(1-a)\eta^2+3-3\eta|\gtrsim |k|^2.    
 \end{align*}
 Then the right side of \eqref{B3} is bounded by
 \begin{align*}
  \norm{\sum_{k_1+k_2=k}\langle k\rangle^{s_1-1}|u_{k_1}||v_{k_2}|} _{l^2_k}\lesssim \norm{u}_{H^s_x}\norm{v}_{H^s_x}  
 \end{align*}
  as long as $s_1-s\leq 1$.
\end{proof}
Writing the equations in \eqref{estimatefourierhir} in the space side and then using the estimates in Propositions \ref{hirprop01}--\ref{hirprop03}, we arrive at
\begin{multline}\label{estimatehirafter1}
\norm{u(t)-e^{-at\partial_x^3}u_0}_{H^{s_1}}\lesssim \norm{u_0}_{H^s}^2+\norm{v_0}_{H^s}^2+\norm{u}_{H^s}^2+\norm{v}_{H^s}^2+\int_0^t\norm{u(r)}_{H^s}^2+\norm{v(r)}_{H^s}^2\text{d}r\\+\norm{\int_0^te^{-a(t-r)\partial_x^3}\big[R_1(u,v,v)(r)+R_2(u,u,u)(r)+R_3(u,v,v)(r)\big]\text{d}r}_{H^{s_1}}  
\end{multline}
and
\begin{multline}\label{estimatehirafter2}
\norm{v(t)-e^{-t\partial_x^3}v_0}_{H^{s_1}}\lesssim \norm{u_0}_{H^s}\norm{v_0}_{H^s}+\norm{u}_{H^s}\norm{v}_{H^s}+\int_0^t\norm{u(r)}_{H^s}\norm{v(r)}_{H^s}\text{d}r\\+\norm{\int_0^te^{-(t-r)\partial_x^3}\big[R_4(u,u,v)(r)+\frac{\beta}{3a}R_4(v,v,v)(r)+R_5(u,u,v)(r)\big]\text{d}r}_{H^{s_1}}  
\end{multline}
Let $\delta$ be the local existence time coming from the local existence theory for the Hirota-Satsuma system. Let $\psi_{\delta}(t)=\psi(t/\delta)$ where $\psi$ is a compactly supported function supported on $[-2,2]$ and $\psi=1$ on $[-1,1]$. For $t\in[-\delta,\delta]$, to estimate the $H^{s_1}$ norms of the integral parts in \eqref{estimatehirafter1}, \eqref{estimatehirafter2}, we need the following standard lemma, see \cite{ginibre}.  
\begin{lemma}\label{standardlemmahir}
For $\frac{1}{2}<b\leq 1$, and $\alpha\neq 0$ \begin{align*}
\norm{\psi_{\delta}(t)\int_0^te^{-\alpha\partial_x^3(t-r)}F(r)\,\text{d}r}_{X_{\alpha}^{s,b}} \lesssim \norm{F}_{X_{\alpha,\delta}^{s,b-1}}.   
\end{align*} 
\end{lemma}
Therefore by the Lemma \ref{standardlemmahir} and the embedding $X_{\alpha,\delta}^{s_1,b}\hookrightarrow L^{\infty}_{t\in[-\delta,\delta]}H^{s_1}_x$ for $b>\frac{1}{2}$, $\alpha\neq 0$, we have
\begin{multline}\label{Restimatehir1}
\norm{\int_0^te^{-a(t-r)\partial_x^3}\big[R_1(u,v,v)(r)+R_2(u,u,u)(r)+R_3(u,v,v)(r)\big]\text{d}r}_{L^{\infty}_{t\in[-\delta,\delta]}{H^{s_1}_x}}\\ \lesssim \norm{\psi_{\delta}(t)\int_0^te^{-a(t-r)\partial_x^3}\big[R_1(u,v,v)(r)+R_2(u,u,u)(r)+R_3(u,v,v)(r)\big]\text{d}r}_{X^{s_1,b}_a}\\ \lesssim \norm{R_1(u,v,v)}_{X^{s_1,b-1}_{a,\delta}}+\norm{R_2(u,u,u)}_{X^{s_1,b-1}_{a,\delta}}+\norm{R_3(u,v,v)}_{X^{s_1,b-1}_{a,\delta}}
\end{multline}
and similarly
\begin{multline}\label{Restimatehir2}
\norm{\int_0^te^{-(t-r)\partial_x^3}\big[R_4(u,u,v)(r)+\frac{\beta}{3a}R_4(v,v,v)(r)+R_5(u,u,v)(r)\big]\text{d}r}_{L^{\infty}_{t\in[-\delta,\delta]}{H^{s_1}_x}}\\ \lesssim \norm{R_4(u,u,v)}_{X^{s_1,b-1}_{1,\delta}}+\norm{R_4(v,v,v)}_{X^{s_1,b-1}_{1,\delta}}+\norm{R_5(u,u,v)}_{X^{s_1,b-1}_{1,\delta}}.\end{multline}
The following estimates for $R_j$, $j=1,2,3,4,5$ are necessary so as to close the argument. Their proofs will be given later on.
\begin{proposition}\label{hirprop1}
  Assume that $u\in \dot{H}^s$. For $s>\frac{1}{2}$, $b-\frac{1}{2}>0$ sufficiently small and $s_1-s<\min\{1, s-\frac{1}{2}, s+2-\mu, 2s+1-\mu\}$, we have 
  \begin{align}\label{R_1}
   \norm{R_1(u,v,w)}_{X_a^{s_1,b-1}}\lesssim \norm{u}_{X_a^{s,1/2}}\norm{v}_{X_1^{s,1/2}}\norm{w}_{X_1^{s,1/2}}.   
 \end{align} When $a=\frac{3p(p-q)+q^2}{q^2}$, $(p,q)\in \mathbb{Z}\times\mathbb{Z}$ with $\frac{q}{2}<p<q$, we replace the above requirement by $s_1-s\leq \min\{1,s-\frac{1}{2}\}$.
 \end{proposition}

\begin{proposition}\label{hirprop2}
  Assume that $u\in \dot{H}^s$. For $s>\frac{1}{2}$, $b-\frac{1}{2}>0$ sufficiently small and $s_1-s\leq 1-$, we have 
  \begin{align*}
   \norm{R_2(u,v,w)}_{X_a^{s_1,b-1}}\lesssim \norm{u}_{X_a^{s,1/2}}\norm{v}_{X_a^{s,1/2}}\norm{w}_{X_a^{s,1/2}}.   
 \end{align*} 
 \end{proposition}

\begin{proposition}\label{hirprop3}
  Assume that $u\in \dot{H}^s$. For $s>\frac{1}{2}$, $b-\frac{1}{2}>0$ sufficiently small and $s_1-s\leq 1$, we have 
\begin{align*}
\norm{R_3(u,v,w)}_{X_a^{s_1,b-1}}\lesssim \norm{u}_{X_a^{s,1/2}}\norm{v}_{X_1^{s,1/2}}\norm{w}_{X_1^{s,1/2}}.\end{align*} 
 \end{proposition}
 
 \begin{proposition}\label{hirprop4}
  Assume that $u\in \dot{H}^s$. For $s>\frac{1}{2}$, $b-\frac{1}{2}>0$ sufficiently small and $s_1-s<\min\{1, s+2-\mu, 2s+1-\mu\}$, we have 
  $$\norm{R_4(u,u,v)}_{X_1^{s_1,b-1}}\lesssim \norm{u}^2_{X_a^{s,1/2}}\norm{v}_{X_1^{s,1/2}}.$$  When $a=\frac{3p(p-q)+q^2}{q^2}$, $(p,q)\in \mathbb{Z}\times\mathbb{Z}$ with $\frac{q}{2}<p<q$, we replace the above requirement by $s_1-s\leq 1$.
 \end{proposition}
 
\begin{proposition}\label{hirprop4ek}
  For $s>\frac{1}{2}$, $b-\frac{1}{2}>0$ sufficiently small and $s_1-s<\min\{1, s+\frac{5}{2}-\mu, 2s+1-\mu\}$, we have
  $$\norm{R_4(u,v,w)}_{X_1^{s_1,b-1}}\lesssim \norm{u}_{X_1^{s,1/2}}\norm{v}_{X_1^{s,1/2}}\norm{w}_{X_1^{s,1/2}}.$$  When $a=\frac{3p(p-q)+q^2}{q^2}$, $(p,q)\in \mathbb{Z}\times\mathbb{Z}$ with $\frac{q}{2}<p<q$, we replace the above requirement by $s_1-s\leq 1$.
 \end{proposition}

\begin{proposition}\label{hirprop5}
  Assume that $u\in \dot{H}^s$. For $s>\frac{1}{2}$, $b-\frac{1}{2}>0$ sufficiently small and $s_1-s<\min\{1, s-\frac{1}{2}, s+\frac{5}{2}-\mu, 2s+1-\mu\}$, we have 
  $$\norm{R_5(u,u,v)}_{X_1^{s_1,b-1}}\lesssim \norm{u}^2_{X_a^{s,1/2}}\norm{v}_{X_1^{s,1/2}}.$$  When $a=\frac{3p(p-q)+q^2}{q^2}$, $(p,q)\in \mathbb{Z}\times\mathbb{Z}$ with $\frac{q}{2}<p<q$, we replace the above requirement by $s_1-s\leq\min\{1-, s-\frac{1}{2}\}$.
 \end{proposition}
 
 Using \eqref{Restimatehir1} and \eqref{Restimatehir2} together with the Propositions \ref{hirprop1}--\ref{hirprop5} in \eqref{estimatehirafter1} and \eqref{estimatehirafter2}, we have
 \begin{multline*}
\norm{u(t)-e^{-at\partial_x^3}u_0}_{H^{s_1}}+\norm{v(t)-e^{-t\partial_x^3}v_0}_{H^{s_1}}\lesssim \big(\norm{u_0}_{H^s}+\norm{v_0}_{H^s}\big)^2\\+ \big(\norm{u(t)}_{H^s}+\norm{v(t)}_{H^s}\big)^2+\int_0^t\big(\norm{u(r)}_{H^s}+\norm{v(r)}_{H^s}\big)^2\,\text{d}r+\big(\norm{u}_{X^{s,1/2}_{a,\delta}}+\norm{v(t)}_{X^{s,1/2}_{1,\delta}}\big)^3 \end{multline*} 
Next we shall obtain the polynomial growth bound stated in the theorem. To do so fix $t$ large. Let $T(r)=\langle r \rangle^{\beta(s)}$. For $r\leq t$, we have that \begin{align*}
    \norm{u(r)}_{H^s}+\norm{v(r)}_{H^s}\lesssim T(t).
\end{align*}
Therefore, for $\delta\approx T(t)^{-\frac{3}{2}}$ and any $j\leq t/\delta\approx tT(t)^{\frac{3}{2}}$, 
\begin{align*}
\norm{u(j\delta)-e^{-\delta a\partial_x^3}u((j-1)\delta)}_{H^{s_1}}+\norm{v(j\delta)-e^{-\delta\partial_x^3}v((j-1)\delta)}_{H^{s_1}} \lesssim T(t)^3   
\end{align*}
where we have used the local theory bound
\begin{align*}
\norm{u}_{X^{s,1/2}_{a,[(j-1)\delta, j\delta]}}+ \norm{v}_{X^{s,1/2}_{1,[(j-1)\delta, j\delta]}}\lesssim \norm{u((j-1)\delta)}_{H^s}\lesssim T(t).  
\end{align*}
Letting $J=t/\delta\approx tT(t)^{\frac{3}{2}}$ yields that
\begin{align*}
\norm{u(J\delta)-e^{-J\delta a\partial_x^3}u_0 }_{H^{s_1}}&\leq \sum_{j=1}^J\norm{e^{-(J-j)\delta a\partial_x^3}u(j\delta)-e^{-(J-j+1)\delta a\partial_x^3}u((j-1)\delta)}_{H^{s_1}}\\&= \sum_{j=1}^J\norm{u(j\delta)-e^{-\delta a\partial_x^3}u((j-1)\delta)}_{H^{s_1}}\lesssim JT(t)^3\approx tT(t)^{9/2}.
\end{align*}
The similar estimate gives the same bound for $v$ completing the demonstration of the growth bound. The continuity in $H^{s_1}\times H^{s_1}$ follows from the continuity of $u$ and $v$ in $H^s$, the embedding $X^{s,b}_a,X^{s,b}_1\hookrightarrow C_t^0H^s_X$, and the estimates stated above, see \cite{erdogan}.

\section{Proofs of Smoothing Estimates}
\subsection{Proof of Proposition \ref{hirprop1}}
 We start by defining the functions:
\begin{align*}
  &f_1(k,\tau)=\langle k\rangle^s \langle\tau-ak^3 \rangle^{\frac{1}{2}}|\widehat{u}_k(\tau)|,\\& f_2(k,\tau)=\langle k\rangle^s \langle\tau-k^3 \rangle^{\frac{1}{2}}|\widehat{v}_k(\tau)|,\\& f_3(k,\tau)=\langle k\rangle^s \langle\tau-k^3 \rangle^{\frac{1}{2}}|\widehat{w}_k(\tau)|. 
\end{align*}
Therefore using these functions, the convolution structure suggest to prove that
\begin{align}\label{aternative}
    \norm{\int\limits_{\sum\tau_j=\tau}\sum\limits_{\sum k_j=k}^*Mf_1(k_1,\tau_1)f_2(k_2,\tau_2)f_3(k_3,\tau_3)}_{\ell_k^2L^2_{\tau}}^2\lesssim \prod_{j=1}^3 \norm{f_j}^2_{\ell_k^2L^2_{\tau}},
\end{align}
where \begin{align*}
M=\frac{|k_2|\langle k\rangle^{s_1}\langle k_1 \rangle^{-s}\langle k_2 \rangle^{-s}\langle k_3\rangle^{-s}}{|k_1+k_2-r_1k||k_1+k_2-r_2k|\langle \tau-ak^3\rangle^{1-b}\langle \tau_1-ak_1^3\rangle^{1/2}\langle\tau_2-k_2^3 \rangle^{1/2}\langle \tau_3-k_3^3 \rangle^{1/2}}.    
\end{align*}
By the Cauchy-Schwarz inequality in $\tau_1$, $\tau_2$, $k_1$, $k_2$ variables, and the application of Young's inequality, the norm in the left hand side of \eqref{aternative} is estimated by
\begin{align*}
    \sup_{k,\tau}\Big(\int\limits_{\sum\tau_j=\tau}\sum\limits_{\sum k_j=k}^*M^2\Big)\norm{f_1^2*f_2^2*f_3^2}_{\ell_k^1L^1_{\tau}}\lesssim \sup_{k,\tau}\Big(\int\limits_{\sum\tau_j=\tau}\sum\limits_{\sum k_j=k}^*M^2\Big)\prod_{j=1}^3 \norm{f_j}^2_{\ell_k^2L^2_{\tau}}.
\end{align*}
Accordingly it suffices to demonstrate that the supremum above is finite. The implementation of the Lemma \ref{cal.lem} in the $\tau_1$ and $\tau_2$ integrals remove the $\tau$ dependence in the supremum and yields a bound
\begin{align*}
    \sup_k\langle k \rangle^{2s_1}\sum_{k_1,k_2}^*\frac{|k_2|^2\langle k_1 \rangle^{-2s}\langle k_2 \rangle^{-2s}\langle k-k_1-k_2 \rangle^{-2s}}{(k_1+k_2-r_1k)^2(k_1+k_2-r_2k)^2\langle ak^3-ak_1^3-k_2^3-(k-k_1-k_2)^3 \rangle^{2-2b}}.
\end{align*}
By a change of variable $k_2\mapsto n-k_1$, it suffices to estimate
\begin{align}\label{mainsup}
    \sup_k\langle k \rangle^{2s_1}\sum_{k_1,n}^*\frac{\langle k_1 \rangle^{-2s}\langle n-k_1 \rangle^{-2s}\langle n-k \rangle^{-2s}|n-k_1|^2}{(n-r_1k)^2(n-r_2k)^2\langle 
    ak^3-ak_1^3-(n-k_1)^3-(k-n)^3 \rangle^{2-2b}}.
 \end{align}
 \\{\bf Case A.} $k_1=k$\\
 In this case, the supremum in \eqref{mainsup} is replaced  by
 \begin{align*}
   \sup_k\langle k \rangle^{2s_1-2s}\sum_{n}^*\frac{\langle n-k \rangle^{-4s}|n-k|^2}{(n-r_1k)^2(n-r_2k)^2}.
\end{align*}
\\{\bf Case A.1.} $|n-r_1k|\geq \delta |k|$, $|n-r_2k|\geq \delta |k|$\\
Note that $|n|\leq \Big[\frac{r_j+\delta}{\delta}\Big]|n-r_jk|$, $j=1,2$. Therefore, in this region $|n-k|\lesssim |n-r_jk|$, $j=1,2$. Then the supremum is bounded by
\begin{align*}
    \sup_k\langle k \rangle^{2s_1-2s-2}\sum_{n}^*\langle n-k \rangle^{-4s}\lesssim \sup_k\langle k\rangle^{2s_1-2s-2}\lesssim 1,   
\end{align*}
for $s_1-s\leq 1$.
\\{\bf Case A.2.} $\delta\leq|n-r_1k|< \delta |k|$ or $\delta \leq|n-r_2k|<\delta |k|$\\
Assume that $\delta\leq|n-r_1k|< \delta |k|$, the other case is similar. Notice that since $|n-k|<(1-r_1)|k|+\delta|k|$, we have $|n-k|\lesssim |k|$. Also the estimate $|n-r_2k|\gtrsim |k|$ follows from $|n-r_2k|\geq (r_1-r_2)|k|-|n-r_1k|\geq (r_1-r_2)|k|-\delta$. As a result, for small but fixed $\delta>0$, using the Lemma \ref{cal.lem}, the supremum is majorized by
\begin{align*}
   \sup_k\langle k \rangle^{2s_1-2s}\sum_{n}^*\frac{\langle n-k \rangle^{-4s}}{\langle n-r_1k \rangle^2}\lesssim \sup_k\langle k \rangle^{2s_1-2s}\langle (1-r_1)k\rangle^{-2}\lesssim \langle k \rangle^{2s_1-2s-2}\lesssim 1
\end{align*}
provided that $s_1-s\leq 1$.
\\{\bf Case A.3.} $|n-r_1k|< \delta$ or $|n-r_2k|<\delta$\\
Suppose that $|n-r_1k|<\delta$, the other case can be dealt in the same way. We note that $|n-k|\gtrsim |k|$, $|n-r_2k|\gtrsim |k|$ and $|n-r_1k|\gtrsim |k|^{1-\mu(r_1)-\epsilon}$. These estimates imply that the supremum above is bounded by
$\sup_k\langle k\rangle^{2s_1-6s-2+2\mu(r_1)+2\epsilon}\lesssim 1$
whenever $s_1-s<2s+1-\mu(r_1)$.
\\{\bf Case B.} $k_1\neq k$\\
In this case we consider the following cases to show that the supremum \eqref{mainsup} is finite.
\\{\bf Case B.1.} $|n-r_1k|<\delta$ or $|n-r_2k|<\delta$\\
Assume the first case $|n-r_1k|<\delta$, the other case follows from a similar treatment. In this region, $|n-k|\gtrsim |k|$, $|n-r_2k|\gtrsim |k|$. Via these estimates, the resulting bound for \eqref{mainsup} is as follows
\begin{align}\label{sup23}
\sup_k\langle k\rangle^{2s_1-2s-4+2\mu(r_1)+2\epsilon}\underset{\substack{ n\simeq r_1k\\ k_1}}{\sum} \frac{\langle k_1\rangle^{-2s}\langle n-k_1\rangle^{-2s}(n-k_1)^2}{\langle ak^3-ak_1^3-(n-k_1)^3-(k-n)^3\rangle^{2-2b}}.  
\end{align}
\\{\bf Case B.1.1} $|k_1|<\delta |k|$\\
In this region, $(r_1-\delta)|k|-\delta<|n-r_1k|<(r_1+\delta)|k|+\delta$. As $u$ is mean zero, $k_1\neq 0$, hence we may write $k_1=\eta_1k$ for some $|k|^{-1}\leq|\eta_1|<\delta$. Also $n=r_1k+\eta_2$ for some $|\eta_2|<\delta$. Using these we obtain \begin{multline*}
|ak^3-ak_1^3-(n-k_1)^3-(k-n)^3|\\=\Big|(k_1-k)\Big(\eta_1k^2\big((1-a)(1+\eta_1)-3r_1\big)+3\eta_2(2r_1-1-\eta_1)k+3\eta_2^2\Big)\Big|\\\geq |k_1-k|\Big(|k|\big(3r_1-(1-a)(1+\delta)-3\delta(2r_1-1+\delta)\big)-3\delta^2\Big)\gtrsim |k||k_1-k|. 
\end{multline*}
Therefore, \begin{align*}
    \eqref{sup23}\lesssim \sup_k\langle k\rangle^{2s_1-4s-4+2b+2\mu(r_1)+2\epsilon}\sum_{k_1} \langle k_1\rangle^{-2s}\langle k_1-k\rangle^{2b-2}\lesssim \sup_k\langle k\rangle^{2s_1-4s-6+4b+2\mu(r_1)+2\epsilon}
\end{align*}
which is finite so long as $s_1-s<s+2-\mu(r_1)$.
\\{\bf Case B.1.2} $|n-k_1|<\delta |k|$\\
Firstly note that $(r_1-\delta)|k|-\delta<|k_1|<(r_1+\delta)|k|+\delta$. We need to bound: \begin{multline}\label{supso}
    \sup_k\langle k\rangle^{2s_1-4s-2+2\mu(r_1)+2\epsilon}\underset{\substack{ k_1\\n\simeq r_1k }}{\sum} \frac{\langle n-k_1\rangle^{-2s}}{\langle ak^3-ak_1^3-(n-k_1)^3-(k-n)^3\rangle^{2-2b}}
\end{multline} In the case $kk_1<0$, we write $n-k_1=\eta_1k$ and $n=r_1k+\eta_2$ for some $|\eta_1|, |\eta_2|< \delta$ to get
\begin{multline*}
|ak^3-ak_1^3-(n-k_1)^3-(k-n)^3|= |(r_1^3-\eta_1^3)k^3-ak_1^3+3\eta_2(1-r_1)^2k^2-3\eta_2^2(1-r_1)k+\eta_2^3|\\ \geq |(r_1^3-\eta_1^3)k^3-ak_1^3|-3\delta(1-r_1)^2k^2-3\delta^2(1-r_1)|k|-\delta^3\\\geq (r_1^3-\delta^3)|k|^3-3\delta(1-r_1)^2k^2-3\delta^2(1-r_1)|k|-\delta^3 \gtrsim |k|^3 
\end{multline*}
by taking sufficiently small $\delta$. This yields that the supremum is bounded when $s_1-s<s+\frac{5}{2}-\mu(r_1)$:
\begin{multline*}
   \eqref{supso}\lesssim \sup_k\langle k\rangle^{2s_1-4s-8+6b+2\mu(r_1)+2\epsilon}\underset{\substack{ k_1\\n\simeq r_1k }}{\sum} \langle n-k_1\rangle^{-2s}\lesssim \sup_k\langle k\rangle^{2s_1-4s-8+6b+2\mu(r_1)+2\epsilon}\lesssim 1. 
\end{multline*}
Next we consider the case in which $kk_1>0$. By observing that $(n,k,k_1)\mapsto(-n,-k,-k_1)$ is a symmetry for \eqref{supso}, we may assume that $k,k_1>0$. By this assumption and the inequality $|n-r_1k|<\delta$, we must have $n>0$ as well, otherwise we would have $|n-k_1|\geq |n|\simeq r_1|k|$. For this case we just write $n=r_1k+\eta$ for some $|\eta|<\delta$ to obtain
\begin{multline*}
|ak^3-ak_1^3-(n-k_1)^3-(k-n)^3|\\= |k_1-k||(1-a)(k^2+kk_1+k_1^2)+3(r_1k+\eta)^2-(r_1k+\eta)(k_1+k)|\\=|k_1-k||(1-a)k_1^2+(-3r_1+1-a)kk_1+\mathcal{O}(\delta)(k+k_1)+\mathcal{O}(\delta^2)|\gtrsim |k-k_1|k^2 
\end{multline*}
where the last inequality is always valid for $k_1$ satisfying $(r_1-\delta)k-\delta<k_1<(r_1+\delta)k+\delta$ with sufficiently small $\delta$. Since $|k-k_1|\gtrsim 1$,
\begin{multline*}
   \eqref{supso}\lesssim \sup_k\langle k\rangle^{2s_1-4s-2+2\mu(r_1)+2\epsilon}\underset{\substack{ k_1\\n\simeq r_1k }}{\sum} \frac{\langle n-k_1\rangle^{-2s}}{\langle (k-k_1)k^2\rangle^{2-2b}}\\\lesssim \sup_k\langle k\rangle^{2s_1-4s-6+4b+2\mu(r_1)+2\epsilon}\lesssim \sum_{k_1} \frac{\langle k_1-r_1k\rangle^{-2s}}{\langle k_1-k\rangle^{2-2b}}\lesssim \sup_k\langle k\rangle^{2s_1-4s-8+6b+2\mu(r_1)+2\epsilon}\lesssim 1 
\end{multline*}
provided that $s_1-s<s+\frac{5}{2}-\mu(r_1)$.
\\{\bf Case B.1.3} $|n-k_1|\geq\delta |k|$, $|k_1|\geq \delta|k|$\\
Note that $|n-k_1|\leq \Big(\frac{2-r_1}{\delta}+1\Big)|k_1|+\delta$. Since $s>1/2$, we have
\begin{align*}
\eqref{sup23}&\lesssim  \sup_k\langle k\rangle^{2s_1-2s-4+2\mu(r_1)+2\epsilon}\underset{\substack{ n\simeq r_1k\\ k_1}}{\sum} \frac{\langle k_1\rangle^{-2s+1}\langle n-k_1\rangle^{-2s+1}}{\langle ak^3-ak_1^3-(n-k_1)^3-(k-n)^3\rangle^{2-2b}}\\& \lesssim  \sup_k\langle k\rangle^{2s_1-6s-2+2\mu(r_1)+2\epsilon}\underset{\substack{ n\simeq r_1k\\ k_1}}{\sum}\langle ak^3-ak_1^3-(n-k_1)^3-(k-n)^3\rangle^{2b-2}\\&\lesssim\sup_k\langle k\rangle^{2s_1-6s-2+2\mu(r_1)+2\epsilon}\lesssim 1   
\end{align*}
whenever $s_1-s<2s+1-\mu(r_1)$.
\\{\bf Case B.2.} $\delta\leq |n-r_1k|<\delta |k|$ or $\delta\leq |n-r_2k|<\delta |k|$\\
Assume that $\delta\leq |n-r_1k|<\delta |k|$, the other case can be treated in a similar fashion. In this region we note that $|n-r_2k|>(r_1-r_2-\delta)|k|$ which implies $|n-r_2k|\gtrsim |k|$. The other required estimates are $(r_1-\delta)|k|<|n|<(r_1+\delta)|k|$, $|n-k|\gtrsim|k|$. Accordingly we need to bound: 
\begin{align*}
 \eqref{mainsup}\lesssim\sup_k \langle k\rangle^{2s_1-2s-2} \underset{\substack{ k_1\\ |k|/4\leq |n|\leq 2|k|}}{\sum}\frac{\langle k_1 \rangle^{-2s}\langle n-k_1\rangle^{-2s}(n-k_1)^2}{\langle  ak^3-ak_1^3-(n-k_1)^3-(k-n)^3\rangle^{2-2b}} . 
\end{align*}
\\{\bf Case B.2.1.} $|k_1|<\delta |k|$\\
Notice in this case that $|n-k_1|\lesssim |k|$. Hence the supremum above is bounded by
\begin{align*}
    \sup_k\langle k\rangle^{2s_1-2s}\underset{\substack{ k_1\\ |n|\geq|k|/4}}{\sum} \langle k_1\rangle^{-2s}\langle n-k_1\rangle^{-2s}\lesssim \sup_k\langle k\rangle^{2s_1-2s}\sum_{|n|\geq |k|/4} \langle n\rangle^{-2s}\lesssim \sup_k\langle k\rangle^{2s_1-4s+1}\lesssim 1
\end{align*}
provided that $s_1-s\leq s-1/2$.
\\{\bf Case B.2.2.} $|n-k_1|<\delta |k|$\\
The computation for this case is the same as that in the previous case.
\\{\bf Case B.2.3.} $|k_1|\geq \delta |k|$, $|n-k_1|\geq\delta |k|$\\
Here $|n-k_1|\lesssim |k_1|$ which leads to the bound
\begin{multline*}
  \sup_k \langle k\rangle^{2s_1-2s-2} \underset{\substack{ k_1\\ |n|\leq 2|k|}}{\sum}\frac{\langle k_1 \rangle^{-2s+1}\langle n-k_1\rangle^{-2s+1}}{\langle  ak^3-ak_1^3-(n-k_1)^3-(k-n)^3\rangle^{2-2b}}\\\lesssim  \sup_k \langle k\rangle^{2s_1-6s} \underset{\substack{ k_1\\ |n|\leq 2|k|}}{\sum}\langle  ak^3-ak_1^3-(n-k_1)^3-(k-n)^3\rangle^{2b-2}\lesssim\sup_k \langle k\rangle^{2s_1-6s+1} \lesssim 1  
\end{multline*}
for $s_1-s\leq 2s-\frac{1}{2}$.
\\{\bf Case B.3.} $|n-r_1k|\geq \delta |k|$, $|n-r_2k|\geq\delta |k|$\\
In this case, we make use of the inequality $|n-k_1|\leq \Big(\frac{r_1+\delta}{\delta}\Big)|n-r_1k|+|k_1|$ so as to have $|n-k_1|^2\lesssim \langle n-k_1 \rangle\big(|k_1|+|n-r_1k|\big)$.
\\{\bf Case B.3.1.} $|k_1|\geq \delta|k|$, $|n-k_1|\geq \delta|k|$\\
In this region, using the inequality above, the supremum \eqref{mainsup} can be bounded by
\begin{multline*}
\sup_k\langle k\rangle^{2s_1-4s-2}\sum_{k_1,n}\frac{\langle n-k \rangle^{-2s}}{\langle ak^3-ak^3_1-(n-k_1)^3-(k-n)^3 \rangle^{2-2b}}\\ \lesssim \sup_k\langle k\rangle^{2s_1-4s-2}\sum_{n}\langle n-k \rangle^{-2s} \lesssim \sup_k\langle k\rangle^{2s_1-4s-2}\lesssim 1 
\end{multline*}
as long as $s_1-s\leq s+1$.
\\{\bf Case B.3.2.} $|k_1|< \delta|k|$\\
In this case, the inequality $|n-k_1|<\Big(\frac{r_1+2\delta}{\delta}\Big)|n-r_1k|$ gives rise to the bound
\begin{multline*}
\eqref{mainsup}\lesssim \sup_k\langle k\rangle^{2s_1-2}\sum_{k_1,n}\langle k_1 \rangle^{-2s} \langle n-k_1 \rangle^{-2s}\langle n-k\rangle^{-2s} \\ \lesssim \sup_k\langle k\rangle^{2s_1-2}\sum_{n}\langle n \rangle^{-2s} \langle n-k \rangle^{-2s}\lesssim\sup_k\langle k\rangle^{2s_1-2s-2}\lesssim 1
\end{multline*}
for $s_1-s\leq 1$.
\\{\bf Case B.3.3.} $|n-k_1|< \delta|k|$\\
The computation in the preceding case works for this case as well. 

\subsection{Proof of Proposition \ref{hirprop2}}
Following the argument in the proof of Proposition \ref{hirprop1}, we need to show that the supremum 
\begin{align*}
    \sup_k\langle k\rangle^{2s_1}\underset{\substack{ k_1\neq0\\k_2\\(k_1+k_2)(k-k_1)(k-k_2)\neq 0}}{\sum}\frac{\langle k_1\rangle^{-2s}\langle k_2 \rangle^{-2s}\langle k-k_1-k_2\rangle^{-2s}}{|k_1|^2\langle (k-k_1)(k-k_2)(k_1+k_2)\rangle^{2-2b}}
\end{align*}
is finite. By the change of variable $k_2\mapsto n-k_1$, it suffices to show that  
\begin{align}\label{son}
    \sup_k\langle k\rangle^{2s_1}\underset{\substack{ k_1, n}}{\sum}\frac{\langle k_1\rangle^{-2s-2}\langle n-k_1 \rangle^{-2s}\langle n-k\rangle^{-2s}}{\langle n\rangle^{2-2b}\langle k_1-k\rangle^{2-2b}\langle n-k-k_1\rangle^{2-2b}}.
\end{align}
is finite. 
\\{\bf Case A.} $|k_1|\gtrsim |k|$\\
In this region,
\begin{align*}
\eqref{son}\lesssim \sup_k\langle k\rangle^{2s_1-2}\underset{\substack{ k_1, n}}{\sum}\frac{\langle k_1\rangle^{-2s}\langle n-k_1 \rangle^{-2s}\langle n-k\rangle^{-2s}}{\langle k\rangle^{2-2b}} \lesssim \sup_k\langle k\rangle^{2s_1-2s-4+2b}    
\end{align*}
which is finite provided that $s_1-s\leq 2-b$.
\\{\bf Case B.} $|k_1|\ll|k|$\\
In this case, the spremum is finite for $s_1-s\leq 2-2b$:
\begin{multline*}
\eqref{son}\lesssim \sup_k\langle k\rangle^{2s_1-2+2b}\underset{\substack{ k_1, n}}{\sum}\frac{\langle k_1\rangle^{-2s-2}\langle n-k_1 \rangle^{-2s}\langle n-k\rangle^{-2s}}{\langle n\rangle^{2-2b}\langle n-k-k_1\rangle^{2-2b}} \\ \lesssim \sup_k\langle k\rangle^{2s_1-2+2b}\underset{\substack{ k_1, n}}{\sum}\frac{\langle k_1\rangle^{-2s-2}\langle n-k_1 \rangle^{-2s}\langle n-k\rangle^{-2s}}{\langle k+k_1\rangle^{2-2b}}\lesssim  \sup_k\langle k\rangle^{2s_1-2s-4+4b} \lesssim 1.
\end{multline*}

\subsection{Proof of Proposition \ref{hirprop3}}
Proceeding as in the proof of Proposition \ref{hirprop1}, it suffices to show that the supremum
\begin{align*}
     \sup_k \langle k \rangle^{2s_1}\underset{\substack{ k_1\neq 0\\ k_2}}{\sum} \frac{\langle k_1 \rangle^{-2s-2}\langle k_2 \rangle^{-2s}\langle k-k_1-k_2 \rangle^{-2s}}{\langle ak^3-ak_1^3-k_2^3-(k-k_1-k_2)^3\rangle^{2-2b}}.
\end{align*}
is finite, or equivalently, by the change of variable $k_2\mapsto n-k_1$, we shall show that the supremum 
\begin{align}\label{R.3}
     \sup_k \langle k \rangle^{2s_1}\underset{\substack{ k_1\neq 0\\ n}}{\sum}\frac{\langle k_1 \rangle^{-2s-2}\langle n-k_1 \rangle^{-2s}\langle n-k \rangle^{-2s}}{\langle ak^3-ak_1^3-(n-k_1)^3-(k-n)^3\rangle^{2-2b}}.
\end{align}
is finite.
\\{\bf Case A.} $k_1=k$\\
In this case, we have
\begin{align*}
\eqref{R.3}\lesssim \sup_k\langle k\rangle^{2s_1-2s-2}\sum_n\langle n-k \rangle^{-4s} \lesssim \sup_k\langle k\rangle^{2s_1-2s-2} \lesssim 1  
\end{align*}
for $s_1-s\leq 1$.
\\{\bf Case B.} $k_1\neq k$
\\{\bf Case B.1.} $|k_1|> \delta |k|$\\
In this case, \eqref{R.3} is finite provide that $s_1-s\leq 1$:
\begin{align*}
\sup_k \langle k \rangle^{2s_1-2}\underset{\substack{ k_1\neq 0\\ n}}{\sum}\langle k_1\rangle^{-2s} \langle n-k_1\rangle^{-2s}\langle n-k\rangle^{-2s} \lesssim \sup_k\langle k\rangle^{2s_1-2s-2}\lesssim 1 . 
\end{align*}
\\{\bf Case B.2.} $|k_1|\leq\delta |k|$\\
\\{\bf Case B.2.1.} $|n-k_1|\leq \delta |k|$\\
In this case, we have $|n-k|\geq |k_1-k|-|k_1-n|\geq (1-2\delta)|k|$. By writing $k_1=\eta_1k$, $n-k_1=\eta_2k$ for some $|k|^{-1}\leq |\eta_1|\leq \delta$ and $0\leq |\eta_2|\leq \delta$, we obtain 
\begin{multline*}
 |ak^3-ak_1^3-(n-k_1)^3-(k-n)^3| =\Big|k^3\Big((1-\eta_1-\eta_2)^3-a+a\eta_1^3+\eta_2^3\Big)\Big|\\=|k^3[1-a+\mathcal{O}(\delta)]|\gtrsim |k|^3. 
\end{multline*}
Using the bound above we get 
\begin{multline*}
\eqref{R.3}\lesssim \sup_k \langle k \rangle^{2s_1-2s}\underset{\substack{ k_1\neq 0\\ |n|\lesssim |k|}}{\sum}\frac{\langle k_1\rangle^{-2s-2}\langle k_1-n\rangle^{-2s}}{\langle k^3\rangle^{2-2b}}\lesssim \sup_k \langle k \rangle^{2s_1-2s-6+6b}\sum_{k_1,n}\langle k_1\rangle^{-2s-2}\langle k_1-n\rangle^{-2s}\\ \lesssim \sup_k \langle k \rangle^{2s_1-2s-6+6b} \lesssim 1   
\end{multline*}
as long as $s_1-s\leq 3-3b$.
 \\{\bf Case B.2.2.} $|n-k|\leq \delta |k|$\\
In this case, we have $|n-k_1|\gtrsim |k|$. Writing $k_1=\eta_1k$, $n-k=\eta_2k$ for some $|k|^{-1}\leq |\eta_1|\leq \delta$ and $0\leq |\eta_2|\leq \delta$, we get
\begin{multline*}
 |ak^3-ak_1^3-(n-k_1)^3-(k-n)^3| =\Big|k^3\Big((1-\eta_1+\eta_2)^3-a+a\eta_1^3-\eta_2^3\Big)\Big|\\=|k^3[1-a+\mathcal{O}(\delta)]|\gtrsim |k|^3. 
\end{multline*}
Proceeding as in the previous case the supremum \eqref{R.3} can be shown to be finite in this region if $s_1-s\leq 3-3b$.

\subsection{Proof of Proposition \ref{hirprop4}}
Using the arguments of the proof of Proposition \ref{hirprop1} we are left with a supremum \begin{align*}
     \sup_k \langle k \rangle^{2s_1}\sum^*_{k_1,k_2\neq 0} \frac{\langle k_1 \rangle^{-2s}\langle k_2 \rangle^{-2s}\langle k-k_1-k_2 \rangle^{-2s}|k-k_1-k_2|^2|k_1+k_2|^2}{\big(k^3-a(k_1+k_2)^3-(k-k_1-k_2)^3\big)^2\langle k^3-(k-k_1-k_2)^3-ak_1^3-ak_2^3\rangle^{2-2b}}.
\end{align*} By a change of variable $k_2\mapsto n-k_1$, the supremum above takes the form
\begin{align}\label{R4}
     \sup_k \langle k \rangle^{2s_1}\underset{\substack{ n\neq 0\\ k_1\neq 0}}{\sum^*}\frac{\langle k_1 \rangle^{-2s}\langle n-k_1 \rangle^{-2s}\langle n-k \rangle^{-2s}|n-k|^2}{(n-\widetilde{r}_1k)^2(n-\widetilde{r}_2k)^2\langle k^3-ak_1^3-a(n-k_1)^3-(k-n)^3\rangle^{2-2b}}.
\end{align}
Note here that the condition $n\neq 0$ results from the factor $n^2$ appearing in the denominator of the prior sum that is reduced to the one in \eqref{R4}.
 \\{\bf Case A.} $|n-\widetilde{r}_1k|<\delta$ or $|n-\widetilde{r}_2k|<\delta$\\
Assume the first case $|n-\widetilde{r}_1k|<\delta$. Handling the other one is similar. Note that
$$|n-\widetilde{r}_2k|\geq (\widetilde{r}_2-\widetilde{r}_1)|k|-|n-\widetilde{r}_1k|>(\widetilde{r}_2-\widetilde{r}_1)|k|-\delta,$$  $$(\widetilde{r}_1-1)|k|-\delta<|n-k|<(\widetilde{r}_1-1)|k|+\delta$$ and $|n-\widetilde{r}_1k|\geq |k|\frac{K(\widetilde{r}_1,\epsilon)}{|k|^{\mu(\widetilde{r}_1)+\epsilon}}\gtrsim |k|^{1-\mu(\widetilde{r}_1)-\epsilon}$. These estimates imply that
\begin{align*}
\eqref{R4}\lesssim\sup_k\langle k\rangle^{2s_1-2s-2+2\mu(\widetilde{r}_1)+2\epsilon} \underset{\substack{ n\simeq \widetilde{r}_1k\\ k_1\neq 0}}{\sum^*}\frac{\langle k_1\rangle^{-2s}\langle n-k_1 \rangle^{-2s}}{\langle k^3-ak_1^3-a(n-k_1)^3-(k-n)^3  \rangle^{2-2b}}
\end{align*}
\\{\bf Case A.1.} $|k_1|\geq \delta |k|$, $|n-k_1|\geq \delta |k|$\\
The supremum above in this case is bounded by
\begin{align*}
 \sup_k\langle k\rangle^{2s_1-6s-2+2\mu(\widetilde{r}_1)+2\epsilon} \underset{\substack{ n\simeq \widetilde{r}_1k\\ k_1\neq 0}}{\sum^*}\langle k^3-ak_1^3-a(n-k_1)^3-(k-n)^3  \rangle^{2b-2}.   
\end{align*}
Write $n=\widetilde{r}_1k+\eta$ for some $|\eta|<\delta$ to get 
\begin{multline*}
|k^3-ak_1^3-a(n-k_1)^3-(k-n)^3|=|(\widetilde{r}_1k+\eta)[3ak_1^2-3a\widetilde{r}_1k_1k+\mathcal{O}(\delta)(k_1+k)+\mathcal{O}(\delta^2)]|\\\ \gtrsim |3ak_1^2-3a\widetilde{r}_1k_1k+\mathcal{O}(\delta)(k_1+k)+\mathcal{O}(\delta^2)|.
\end{multline*}
Use this estimate along with the second claim of Lemma \ref{cal.lem} for the sum in $k_1$ to conclude that the supremum is finite whenever $s_1-s<2s+1-\mu(\widetilde{r}_1)$. 
\\{\bf Case A.2.} $|k_1|< \delta |k|$\\
Here $|n-k_1|\gtrsim |k|$ since $|n-k_1|>(\widetilde{r}_1-\delta)|k|-\delta$. As above we write $n=\widetilde{r}_1k+\eta$ for some $|\eta|<\delta$ to obtain 
\begin{multline*}
|k^3-ak_1^3-a(n-k_1)^3-(k-n)^3|=|(\widetilde{r}_1k+\eta)[3ak_1(\widetilde{r}_1k-k_1)+\mathcal{O}(\delta)(k_1+k)+\mathcal{O}(\delta^2)]|\\\ \gtrsim |k_1||k|^2.
\end{multline*}
It follows that the supremum is bounded by
\begin{align*}
\sup_k\langle k\rangle^{2s_1-4s-2+2\mu(\widetilde{r}_1)+2\epsilon} 
\sum_{k_1\neq 0}\frac{\langle k_1 \rangle^{-2s-1}}{\langle k \rangle^{4-4b}}\lesssim \sup_k\langle k\rangle^{2s_1-4s-6+4b+2\mu(\widetilde{r}_1)+2\epsilon}
\end{align*}
that is finite only if $s_1-s<s+2-\mu(\widetilde{r}_1)$.
\\{\bf Case A.3.} $|n-k_1|<\delta |k|$\\
Clearly $|k_1|\gtrsim |k|$. Moreover, first writing $n=\widetilde{r}_1k+\eta$ for some $|\eta|<\delta$, and then reinstating the variable $n$ we get
\begin{multline*}
|k^3-ak_1^3-a(n-k_1)^3-(k-n)^3|=|n||3ak_1(n-k_1)+\mathcal{O}(\delta)k+\mathcal{O}(\delta^2)|\gtrsim |k|^2|n-k_1|.\end{multline*}
Therefore, by the mean zero assumption on $u$, $n-k_1\neq 0$, we have
\begin{align*}
\sup_k\langle k\rangle^{2s_1-4s-2+2\mu(\widetilde{r}_1)+2\epsilon} 
\underset{\substack{ n\simeq \widetilde{r}_1k\\ |k_1|\gtrsim|k|}}{\sum}\frac{\langle k_1-n \rangle^{-2s-1}}{\langle k \rangle^{4-4b}}\lesssim \sup_k\langle k\rangle^{2s_1-4s-6+4b+2\mu(\widetilde{r}_1)+2\epsilon}\lesssim 1
\end{align*} provided that $s_1-s<s+2-\mu(\widetilde{r}_1)$.
\\{\bf Case B.} $\delta\leq |n-\widetilde{r}_1k|<\delta |k|$ or $\delta\leq |n-\widetilde{r}_2k|<\delta |k|$\\
Suppose that $\delta\leq |n-\widetilde{r}_1k|<\delta |k|$, the other case is analogous. Notice in this case that $|n-\widetilde{r}_2k|\gtrsim |k|$ since
$|n-\widetilde{r}_2k|\geq (\widetilde{r}_2-\widetilde{r}_1)|k|-|n-\widetilde{r}_1k|>(\widetilde{r}_2-\widetilde{r}_1-\delta)|k|.$ Furthermore,
$(\widetilde{r}_1-1-\delta)|k|<|n-k|<(\widetilde{r}_1-1+\delta)|k|$. So we have
\begin{align*}
\eqref{R4}\lesssim \sup_k\langle k \rangle^{2s_1-2s} \underset{\substack{ n\neq 0\\ k_1\neq 0}}{\sum} \frac{\langle k_1\rangle^{-2s}\langle n-k_1 \rangle^{-2s}}{\langle k^3-ak_1^3-a(n-k_1)^3-(k-n)^3\rangle^{2-2b}}. 
\end{align*}
\\{\bf Case B.1.} $|k_1|\geq \delta |k|$, $|n-k_1|\geq \delta |k|$\\ 
In this region, using Lemma \ref{cal.lem} the supremum is bounded by
\begin{multline*}
\sup_k\langle k \rangle^{2s_1-4s} \underset{\substack{n\neq0 \\ |k_1|\gtrsim |k|}}{\sum} \frac{\langle k_1\rangle^{-2s}}{\langle k^3-ak_1^3-a(n-k_1)^3-(k-n)^3\rangle^{2-2b}}\\ \lesssim \sup_k\langle k \rangle^{2s_1-4s} \underset{\substack{ |k_1|\gtrsim |k|}}{\sum} \langle k_1\rangle^{-2s}\lesssim \sup_k\langle k \rangle^{2s_1-6s+1}\lesssim 1, 
\end{multline*}
when $s_1-s\leq 2s-\frac{1}{2}$.
\\{\bf Case B.2.} $|k_1|<\delta |k|$\\
Notice that $|n-k_1|\gtrsim |k|$ because $|n-k_1|> (\widetilde{r}_1-2\delta)|k|$. We write $n=(\widetilde{r}_1+\eta)k$ for some $|\eta|<\delta$ to attain \begin{multline*}
|k^3-ak_1^3-a(n-k_1)^3-(k-n)^3|=|(\widetilde{r}_1+\eta)k[3ak_1((\widetilde{r}_1+\eta)k-k_1)+\mathcal{O}(\delta)k^2]|\\\geq 3a(\widetilde{r}_1-2\delta)^2|k_1||k|^2\gtrsim |k|^2|k_1|.\end{multline*} 
Thus the supremum is finite when $s_1-s< s+\frac{1}{2}$:
\begin{multline*}
\sup_k\langle k \rangle^{2s_1-4s} \underset{\substack{ |n|\lesssim |k| \\ k_1\neq 0}}{\sum} \frac{\langle k_1\rangle^{-2s}}{\langle k^2k_1\rangle^{2-2b}} \lesssim \sup_k\langle k \rangle^{2s_1-4s-3+4b} \underset{\substack{ |k_1|\lesssim |k|}}{\sum} \langle k_1\rangle^{-2s-2+2b}\lesssim \sup_k\langle k \rangle^{2s_1-4s-3+4b}\lesssim 1.
\end{multline*}
\\{\bf Case B.3.} $|n-k_1|<\delta |k|$\\
In this case $|k_1|\gtrsim |k|$ due to $|k_1|>(\widetilde{r}_1-2\delta)|k|$. Accordingly, as in the previous case, first writing $n=(\widetilde{r}_1+\eta)k$ for some $|\eta|<\delta$ and then reinstating the variable $n$, we have \begin{multline*}
|k^3-ak_1^3-a(n-k_1)^3-(k-n)^3|=\big|(\widetilde{r}_1+\eta)k\big(3ak_1((\widetilde{r}_1+\eta)k-k_1)+\mathcal{O}(\delta)k^2\big)\big|\\\gtrsim |k|^2|n-k_1|.\end{multline*} This, recalling the mean zero assumption on $u$, gives rise to the bound for the supremum
\begin{align*}
\sup_k\langle k \rangle^{2s_1-4s} \underset{\substack{ |n|\lesssim |k| \\ |k_1|\gtrsim |k|}}{\sum} \frac{\langle n-k_1\rangle^{-2s-1}}{\langle k^2\rangle^{2-2b}}\lesssim \sup_k\langle k \rangle^{2s_1-4s-3+4b}\lesssim 1
\end{align*}
on the condition that $s_1-s<s+\frac{1}{2}$.
\\{\bf Case C.} $|n-\widetilde{r}_1k|\geq \delta|k|$, $|n-\widetilde{r}_2k|\geq \delta|k|$\\
We note that $|n|\leq\Big[\frac{\widetilde{r}_j+\delta}{\delta}\Big]|n-\widetilde{r}_jk|$, $j=1,2$. In this region, this implies that $|n-k|\lesssim|n-\widetilde{r}_1k|$. Hence the supremum is finite if $s_1-s\leq 1$:
\begin{multline*}
\eqref{R4}\lesssim \sup_k\langle k\rangle^{2s_1-2} \underset{\substack{ n\neq0 \\ k_1\neq 0}}{\sum} \frac{\langle k_1 \rangle^{-2s}\langle n-k_1\rangle^{-2s}\langle n-k\rangle^{-2s}}{\langle k^3-ak_1^3-a(n-k_1)^3-(k-n)^3\rangle^{2-2b}}\\ \lesssim \sup_k\langle k\rangle^{2s_1-2} \underset{\substack{ n\neq0 \\ k_1\neq 0}}{\sum} \langle k_1 \rangle^{-2s}\langle n-k_1\rangle^{-2s}\langle n-k\rangle^{-2s} \lesssim \sup_k\langle k\rangle^{2s_1-2}\sum_n\langle n\rangle^{-2s}\langle n-k \rangle^{-2s}\\\lesssim \langle k\rangle^{2s_1-2s-2} \lesssim 1.
\end{multline*}

\subsection{Proof of Proposition \ref{hirprop4ek}}

 We are to handle the supremum
 \begin{align*}
     \sup_k \langle k \rangle^{2s_1}\sum^*_{k_1,k_2} \frac{\langle k_1 \rangle^{-2s}\langle k_2 \rangle^{-2s}\langle k-k_1-k_2 \rangle^{-2s}|k-k_1-k_2|^2|k_1+k_2|^2}{(k_1+k_2)^2(k_1+k_2-\widetilde{r}_1k)^2(k_1+k_2-\widetilde{r}_2k)^2\langle (k-k_1)(k-k_2)(k_1+k_2)\rangle^{2-2b}}
 \end{align*} which is equivalent, by a change of variable $k_2\mapsto n-k_1$, to
 \begin{align}\label{R5}
     \sup_k \langle k \rangle^{2s_1}\underset{\substack{ n\neq 0\\ k_1}}{\sum^*}\frac{\langle k_1 \rangle^{-2s}\langle n-k_1 \rangle^{-2s}\langle n-k \rangle^{-2s}|n-k|^2}{(n-\widetilde{r}_1k)^2(n-\widetilde{r}_2k)^2\langle n(k-k_1)(k+k_1-n)\rangle^{2-2b}}.
 \end{align} In the case $n(k-k_1)(k+k_1-n)=0$, that is either $k_1=k$ or $k_1=n-k$, \eqref{R5} boils down to
 \begin{align*}
     \sup_k \langle k \rangle^{2s_1-2s} \sum^*_{n\neq 0} \frac{\langle n-k \rangle^{-4s}|n-k|^2}{(n-\widetilde{r}_1k)^2(n-\widetilde{r}_2k)^2}
 \end{align*}
 which essentially can be treated as that in the Case A of the proof of Proposition \ref{hirprop1}. Hence the supremum is finite if $s_1-s\leq 1$ and $s_1-s<2s+1-\mu(\widetilde{r}_j)$. Next we move to the complementary case:
 \\{\bf Case A.} $n(k-k_1)(k+k_1-n)\neq0$\\
 In this case, \begin{align*}
     \eqref{R5}\lesssim \sup_k \langle k \rangle^{2s_1} \underset{\substack{ n\neq 0\\ k_1}}{\sum^*} \frac{\langle k_1 \rangle^{-2s}\langle n-k_1 \rangle^{-2s}\langle n-k \rangle^{-2s}|n-k|^2}{(n-\widetilde{r}_1k)^2(n-\widetilde{r}_2k)^2\langle n\rangle^{2-2b}\langle n-k-k_1\rangle^{2-2b}\langle k_1-k\rangle^{2-2b}}.
 \end{align*}
 \\{\bf Case A.1.} $|n-\widetilde{r}_1k|\geq \delta |k|$, $|n-\widetilde{r}_2k|\geq \delta |k|$\\
 In this region, $|n|\leq \Big(\frac{\widetilde{r}_j+\delta}{\delta}\Big)|n-\widetilde{r}_jk|$, $j=1,2$. Thus, $|n-k|\lesssim |n-\widetilde{r}_1k|$, by which the supremum above is estimated by
 \begin{multline*}
\sup_k\langle k\rangle^{2s_1-2} \underset{\substack{ n\neq 0\\ k_1}}{\sum} \frac{\langle k_1 \rangle^{-2s}\langle n-k_1 \rangle^{-2s}\langle n-k \rangle^{-2s}}{\langle n\rangle^{2-2b}\langle n-k-k_1\rangle^{2-2b}\langle k_1-k\rangle^{2-2b}}\\ \lesssim  \sup_k\langle k\rangle^{2s_1-2} \underset{\substack{ n\neq 0\\ k_1}}{\sum} \langle k_1 \rangle^{-2s}\langle n-k_1 \rangle^{-2s} \langle n-k \rangle^{-2s}  \lesssim  \sup_k\langle k\rangle^{2s_1-2s-2}\lesssim 1
 \end{multline*}
 for $s_1-s\leq 1$.
 \\{\bf Case A.2.} $\delta \leq |n-\widetilde{r}_1k|<\delta |k|$ or $\delta \leq |n-\widetilde{r}_2k|< \delta |k|$\\
 Assume the case $\delta \leq |n-\widetilde{r}_1k|<\delta |k|$, the other one is treated similarly. In this case, the estimates 
 \begin{align*}
     |n-\widetilde{r}_2k|\geq (\widetilde{r}_2-\widetilde{r}_1)|k|-|n-\widetilde{r}_1k|>(\widetilde{r}_2-\widetilde{r}_1-\delta)|k|,
 \end{align*}
 $(\widetilde{r}_1-\delta) |k|<|n|<(\widetilde{r}_1+\delta)|k|$, and $(\widetilde{r}_1-1-\delta)|k|<|n-k|<(\widetilde{r}_1-1+\delta)|k|$ lead to the bound
 \begin{align*}
 \sup_k\langle k\rangle^{2s_1-2s-2+2b}\underset{\substack{ n\neq 0\\ k_1}}{\sum} \frac{\langle k_1 \rangle^{-2s}\langle n-k_1 \rangle^{-2s}}{\langle k_1-k \rangle^{2-2b}\langle k_1+k-n \rangle^{2-2b}}&\lesssim \sup_k\langle k\rangle^{2s_1-2s-4+4b}\sum_{k_1} \frac{\langle k_1 \rangle^{-2s}}{\langle k_1-k \rangle^{2-2b}}\\& \lesssim \sup_k\langle k\rangle^{2s_1-2s-6+6b} \lesssim 1
\end{align*}
 provided that $s_1-s\leq 3-3b$.
 \\{\bf Case A.3.} $|n-\widetilde{r}_1k|<\delta$ or $|n-\widetilde{r}_2k|< \delta $\\
 Assume that $|n-\widetilde{r}_1k|<\delta$, the treatment of the other case is similar. Note that \begin{align*}\label{sum}
    &(\widetilde{r}_1-1)|k|-\delta<|n-k|<(\widetilde{r}_1-1)|k|+\delta,\\&|n-\widetilde{r}_2k|\geq (\widetilde{r}_2-\widetilde{r}_1)|k|-|n-\widetilde{r}_1k|>(\widetilde{r}_2-\widetilde{r}_1)|k|-\delta .
 \end{align*}
 Therefore the supremum is majorized by 
 \begin{align}
\sup_k\langle k\rangle^{2s_1-2s-4+2b+2\mu(\widetilde{r}_1)+2\epsilon} \underset{\substack{ n\simeq \widetilde{r}_1k\\ k_1}}{\sum}\frac{\langle k_1\rangle^{-2s}\langle n- k_1\rangle^{-2s}}{\langle k-k_1\rangle^{2-2b}\langle n-k-k_1\rangle^{2-2b}}.   
 \end{align}
 There is merely a single term for the sum in $n$, that is the one with $n\simeq \widetilde{r}_1k$. So only the estimate regarding the sum in $k_1$ matters here. If $|k_1|\ll |k|$ then all the other factors in the sum in \eqref{sum} are of order $\gtrsim |k|$; likewise if $|n-k_1|\ll|k|$ then the remaining factors are again of order $\gtrsim |k|$. Thus in these cases, the sum in $\eqref{sum}\lesssim \langle k\rangle^{-4-2s+4b}$ entailing $\eqref{sum}\lesssim\langle k\rangle^{2s_1-4s-8+6b+2\mu(\widetilde{r}_1)+2\epsilon}$ which is finite as long as $2s_1-4s-8+6b+2\mu(\widetilde{r}_1)+2\epsilon\leq 0$ or equivalently $s_1-s<s+5/2-\mu(\widetilde{r}_1)$. If $|k_1-k|\ll |k|$ then the factors with exponent $-2s$ in the numerator are of order $\gtrsim |k|$; likewise if $|n-k-k_1|\ll |k|$ then the factors in the numerator are of order $\gtrsim |k|$. In the either case, the sum in $\eqref{sum}\lesssim \langle k\rangle^{-4s-3+4b}$ giving rise to $\eqref{sum}\lesssim\langle k\rangle^{2s_1-6s-7+6b+2\mu(\widetilde{r}_1)+2\epsilon}\lesssim 1$ provided that $s_1-s<2s+2-\mu(\widetilde{r}_1)$.

\subsection{Proof of Proposition \ref{hirprop5}} 
In order to handle $R_5$, we need to divide the sum into pieces where $k_1+k_2\neq 0$ and $k_1+k_2=0$.
\begin{multline*}
R_5(u,u,v)_k=-9kv_k\sum\limits_{k_1\neq 0}^*\frac{(k-k_1)|u_{k_1}|^2}{k^3-ak_1^3-(k-k_1)^3}+9i \underset{\substack{k_1+k_2+k_3=k \\ k_1+k_2\neq 0\\ k_1\neq 0}}{\sum^*} \frac{k_3(k_2+k_3)u_{k_1}u_{k_2}v_{k_3}}{k^3-ak_1^3-(k_2+k_3)^3} \\=-9kv_k\sum\limits_{k_1>0}^*\Bigg(\frac{k+k_1}{k^3+ak_1^3-(k+k_1)^3}+\frac{k-k_1}{k^3-ak_1^3-(k-k_1)^3}\Bigg)|u_{k_1}|^2\\+9i \underset{\substack{k_1+k_2+k_3=k \\ k_1+k_2\neq 0\\ k_1\neq 0}}{\sum^*} \frac{k_3(k_2+k_3)u_{k_1}u_{k_2}v_{k_3}}{k^3-ak_1^3-(k_2+k_3)^3}
\\=18kv_k\sum\limits_{k_1>0}^*\frac{k_1^2|u_{k_1}|^2}{(1-a)(k_1-\widetilde{r}_1k)(k_1+\widetilde{r}_1k)(k_1-\widetilde{r}_2k)(k_1+\widetilde{r}_2k)}\\+9i \underset{\substack{k_1+k_2+k_3=k \\ k_1+k_2\neq 0\\ k_1\neq 0}}{\sum^*} \frac{k_3(k_2+k_3)u_{k_1}u_{k_2}v_{k_3}}{k^3-ak_1^3-(k_2+k_3)^3}=:S_1+S_2
\end{multline*}
For the first sum, using Cauchy-Schwarz and Young inequalities and the Lemma \ref{cal.lem} yields that, 
\begin{multline*}\norm{S_1}_{X_1^{s_1,b-1}}\lesssim \sup_k\langle k \rangle^{2s_1+2-2s}\sum\limits_{k_1>0}^*\frac{\langle k_1\rangle^{4-4s}}{(k_1-\widetilde{r}_1k)^2(k_1+\widetilde{r}_1k)^2(k_1-\widetilde{r}_2k)^2(k_1+\widetilde{r}_2k)^2}\\ \times\norm{u}^2_{X_a^{s,1/2}}\norm{v}_{X_1^{s,1/2}}.\end{multline*}
Thus it is required to show that the supremum above is finite. Since $k_1>0$ and $\widetilde{r}_1,\widetilde{r}_2>0$, to take advantage of the multipliers in the denominator of the sum in the supremum, we consider the cases in which $k<0$ and $k>0$. We just examine the $k<0$ case as the other case can be treated similarly. Thus, by the sign considerations, both $|k_1-\widetilde{r}_1k|$ and $|k_1-\widetilde{r}_2k|$ are of order $\gtrsim |k|, k_1$ by which the supremum is replaced by the bound \begin{align}\label{6}\sup_k\langle k\rangle^{2s_1-2s}\sum\limits_{k_1>0}^*\frac{\langle k_1\rangle^{2-4s}}{(k_1+\widetilde{r}_1k)^2(k_1+\widetilde{r}_2k)^2}.\end{align} First we observe that the case $|k_1+\widetilde{r}_1k|, |k_1+\widetilde{r}_2k|\leq \delta |k|$ cannot arise concurrently, because choosing $\delta < \frac{\widetilde{r}_2-\widetilde{r}_1}{2}$ entails that
$$(\widetilde{r}_2-\widetilde{r}_1)|k|\leq |k_1+\widetilde{r}_1k|+|k_1+\widetilde{r}_2k|\leq 2\delta |k|<(\widetilde{r}_2-\widetilde{r}_1)|k|.$$ We consider the following cases:\\
{\bf Case A.} $|k_1+\widetilde{r}_1k|, |k_1+\widetilde{r}_2k|\geq \delta |k|$\\
In this case, $|k_1+\widetilde{r}_jk|\geq \Big(\frac{\delta}{\delta+\widetilde{r}_j}\Big)k_1$, $j=1,2$, that implies 
\begin{align*}
\eqref{6} \lesssim \sup_k\langle k \rangle^{2s_1-2s-2}\sum\limits_{k_1>0}\langle k_1\rangle^{-4s}  \end{align*}
 which is finite provided that $s_1-s\leq 1$.
\\{\bf Case B.} $|k_1+\widetilde{r}_1k|\geq \delta |k|$, $\delta \leq |k_1+\widetilde{r}_2k|<\delta |k|$ (or with the roles of $\widetilde{r}_1$ and $\widetilde{r}_2$ are switched)
Note that $(\widetilde{r}_2-\delta)|k|<k_1<(\widetilde{r}_2+\delta)|k|$. Then the supremum is bounded by
$$\sup_k\langle k\rangle^{2s_1-2s-1}\sum\limits_{k_1\geq |k|}\langle k_1 \rangle^{-4s+1}\lesssim \sup_k\langle k\rangle^{2s_1-6s+1}\lesssim 1$$
for $s_1-s\leq 2s-\frac{1}{2}$.
\\{\bf Case C.} $|k_1+\widetilde{r}_1k|\geq \delta |k|$, $|k_1+\widetilde{r}_2k|<\delta $ (or with the roles of $\widetilde{r}_1$ and $\widetilde{r}_2$ are switched)\\ Using the bound $|k_1+\widetilde{r}_2k|\gtrsim |k|^{1-\mu(\widetilde{r}_2)-\epsilon}$ and $k_1\simeq -\widetilde{r}_2k$, 
$$\eqref{6}\lesssim \langle k \rangle^{2s_1-6s-2+2\mu(\widetilde{r}_2)+2\epsilon}\lesssim 1$$
as long as $s_1-s<2s+1-\mu(\widetilde{r}_2)$.
As for the $X^{s_1,b-1}_1$ norm of the sum $S_2$, proceeding as before, we need to show that
$$\sup_k\langle k\rangle^{2s_1}\underset{\substack{k_1\neq 0, k_2 \\ k_1+k_2\neq 0}}{\sum}\frac{\langle k_1 \rangle^{-2s}\langle k_2\rangle^{-2s}\langle k-k_1-k_2 \rangle^{-2s}|k-k_1-k_2|^2|k-k_1|^2}{\big[k^3-ak_1^3-(k_2+k_3)^3\big]^2\langle k^3-ak_1^3-ak_2^3-(k-k_1-k_2)^3\rangle^{2-2b}}\lesssim 1.$$ This, by the change of variable $k_2\mapsto n-k_1$, is equivalent to estimate 
\begin{align}\label{7}\sup_k\langle k\rangle^{2s_1}\underset{\substack{k_1\neq 0\\ n\neq 0}}{\sum}\frac{\langle k_1 \rangle^{-2s-2}\langle n-k_1\rangle^{-2s}\langle n-k \rangle^{-2s}|n-k|^2|k-k_1|^2}{(k_1-\widetilde{r}_1k)^2(k_1-\widetilde{r}_2k)^2\langle k^3-ak_1^3-a(n-k_1)^3-(k-n)^3\rangle^{2-2b}}.\end{align}
{\bf Case A.} $|k_1-\widetilde{r}_1k|<\delta$ or $|k_1-\widetilde{r}_2k|<\delta$ \\
The treatment of the both cases are similar, so assume that $|k_1-\widetilde{r}_1k|<\delta$. We have the following estimates 
$$
|k_1-k|\leq (\widetilde{r}_1-1)|k|+|k_1-\widetilde{r}_1k|< (\widetilde{r}_1-1)|k|+\delta, $$ $$|k_1-\widetilde{r}_2k|\geq (\widetilde{r}_2-\widetilde{r}_1)|k|-|k_1-\widetilde{r}_1k|> (\widetilde{r}_2-\widetilde{r}_1)|k|-\delta.$$ 
{\bf Case A.1.} $|n-k_1|\geq \delta |k|$, $|n-k|\geq \delta |k|$ \\ Using the inequality $|n-k|\leq |n-k_1|+|k_1-k|$, the relation $-2s+1<0$ and the above estimates,
\begin{multline*}
 \eqref{7}\lesssim \sup_k\langle k \rangle^{2s_1}\underset{\substack{k_1\simeq \widetilde{r}_1k\\ n\neq 0}}{\sum}\frac{\langle k_1 \rangle^{-2s-2}\langle n-k_1 \rangle^{-2s+1}\langle n-k \rangle^{-2s+1}}{(k_1-\widetilde{r}_1k)^2\langle k^3-ak_1^3-a(n-k_1)^3-(k-n)^3\rangle^{2-2b}}\\+\sup_k\langle k \rangle^{2s_1+1}\underset{\substack{k_1\simeq \widetilde{r}_1k\\ n\neq 0}}{\sum}\frac{\langle k_1 \rangle^{-2s-2}\langle n-k_1 \rangle^{-2s}\langle n-k \rangle^{-2s+1}}{(k_1-\widetilde{r}_1k)^2\langle k^3-ak_1^3-a(n-k_1)^3-(k-n)^3\rangle^{2-2b}}\\ \lesssim \sup_k\langle k\rangle^{2s_1-6s-2+2\mu(\widetilde{r}_1)+2\epsilon }\underset{\substack{k_1\simeq \widetilde{r}_1k\\ n\neq 0}}{\sum}\frac{1}{\langle k^3-ak_1^3-a(n-k_1)^3-(k-n)^3\rangle^{2-2b}}
\end{multline*}
which is finite provided that $s_1-s<2s+1-\mu(\widetilde{r}_1)$. 
\\{\bf Case A.2.} $|n-k|< \delta|k|$\\ Here $|n-k_1|\geq (\widetilde{r}_1-1)|k|-|k-n|-|k_1-\widetilde{r}_1k|>(\widetilde{r}_1-1-\delta)|k|-\delta.$ In this region for $|\eta_1|, |\eta_2|<\delta$, we may write
$n-k=\eta_1k$ and $k_1-\widetilde{r}_1k=\eta_2$. So we have
\begin{multline*}
|k^3-ak_1^3-a(n-k_1)^3-(k-n)^3|=|k^3-a(\widetilde{r}_1k+\eta_2)^3-a\big((1+\eta_1-\widetilde{r}_1)k-\eta_2\big)^3+\eta_1^3k^3|\\=|\big(1-a+3a\widetilde{r}_1-3a\widetilde{r}_1^2+\mathcal{O}(\delta)\big)k^3+\mathcal{O}(\delta)k^2+\mathcal{O}(\delta^2)k+\mathcal{O}(\delta^3)|\gtrsim |k|^3,   
\end{multline*}
the last inequality follows since $\widetilde{r}_1$ is the root of the quadratic $(1-a)x^2-3x+3$. Using these bounds 
$$\eqref{7}\lesssim \sup_k\langle k \rangle^{2s_1-4s-2+2\mu(\widetilde{r}_1)+2\epsilon}\underset{\substack{ n\neq 0}}{\sum}\frac{\langle n-k\rangle^{-2s}}{\langle  k^3\rangle^{2-2b}}\lesssim \sup_k\langle k \rangle^{2s_1-4s-8+6b+2\mu(\widetilde{r}_1)+2\epsilon}\lesssim 1$$ for $s_1-s <s+4-3b-\mu(\widetilde{r}_1)$.
\\{\bf Case A.3.} $|n-k_1|< \delta|k|$\\
In this region, $(\widetilde{r}_1-1-\delta)|k|-\delta<|n-k|<(\widetilde{r}_1-1+\delta)|k|+\delta$. So we may write $n-k=\eta_1k$ for some $\eta_1$ with $|k|^{-1}\leq|\eta_1|\leq \epsilon$ and $k_1-\widetilde{r}_1k=\eta_2$ for some $\eta_2$ with $|\eta_2|<\delta<\epsilon$. Therefore
\begin{multline*}
|k^3-ak_1^3-a(n-k_1)^3-(k-n)^3|=|k^3-a(\widetilde{r}_1k+\eta_2)^3-a\big((1+\eta_1-\widetilde{r}_1)k-\eta_2\big)^3+\eta_1^3k^3|\\=|\big(1-a+3a\widetilde{r}_1-3a\widetilde{r}_1^2+\mathcal{O}(\epsilon)\big)k^3+\mathcal{O}(\delta)k^2+\mathcal{O}(\delta^2)k+\mathcal{O}(\delta^3)|\gtrsim |k|^3,   
\end{multline*}
it follows, as in the previous case, that the supremum is bounded for $s_1-s <s+4-3b-\mu(\widetilde{r}_1)$:
$$\eqref{7}\lesssim \sup_k\langle k \rangle^{2s_1-4s-2+2\mu(\widetilde{r}_1)+2\epsilon}\underset{\substack{ n\neq 0}}{\sum}\frac{\langle n-k_1\rangle^{-2s}}{\langle  k^3\rangle^{2-2b}}\lesssim \sup_k\langle k \rangle^{2s_1-4s-8+6b+2\mu(\widetilde{r}_1)+2\epsilon}\lesssim 1.$$
\\{\bf Case B.} $\delta\leq |k_1-\widetilde{r}_1k|< \delta|k|$ or $\delta\leq |k_1-\widetilde{r}_2k|< \delta|k|$ \\
We assume the first case $\delta\leq |k_1-\widetilde{r}_1k|< \delta|k|$; the second one can be treated in a similar fashion. In this region, we have the estimates: $|k_1-k|<(\widetilde{r}_1-1+\delta)|k|$, $|k_1-\widetilde{r}_2k|\geq (\widetilde{r}_2-\widetilde{r}_1)|k|-|k_1-\widetilde{r}_1k|>(\widetilde{r}_2-\widetilde{r}_1-\delta)|k|$. Also $|k_1-\widetilde{r}_1k|< \delta|k|$ implies $|k_1|>|k|$. Thus, 
\begin{align}\label{8}
    \eqref{7}\lesssim \sup_k\langle k \rangle^{2s_1}\underset{\substack{|k_1|>|k|\\ n\neq 0}}{\sum}\frac{\langle k_1 \rangle^{-2s-2}\langle n-k_1 \rangle^{-2s}\langle n-k \rangle^{-2s}|n-k|^2}{\langle k^3-ak_1^3-a(n-k_1)^3-(k-n)^3\rangle^{2-2b}}.
\end{align}
\\{\bf Case B.1.} $|n-k_1|\geq \delta|k|$, $|n-k|\geq \delta|k|$\\
In this case,
\begin{multline*}
    \eqref{8}\lesssim \sup_k\langle k \rangle^{2s_1}\underset{\substack{|k_1|>|k|\\ n\neq 0}}{\sum}\frac{\langle k_1 \rangle^{-2s-2}\langle n-k_1 \rangle^{-2s+1}\langle n-k \rangle^{-2s+1}}{\langle k^3-ak_1^3-a(n-k_1)^3-(k-n)^3\rangle^{2-2b}}\\+\sup_k\langle k \rangle^{2s_1}\underset{\substack{|k_1|>|k|\\ n\neq 0}}{\sum}\frac{\langle k_1 \rangle^{-2s-2}\langle n-k_1 \rangle^{-2s}\langle n-k \rangle^{-2s+1}|k_1-k|}{\langle k^3-ak_1^3-a(n-k_1)^3-(k-n)^3\rangle^{2-2b}}\\ \lesssim \sup_k\langle k \rangle^{2s_1-4s}\underset{\substack{|k_1|>|k|\\ n\neq 0}}{\sum}\frac{\langle k_1 \rangle^{-2s}}{\langle k^3-ak_1^3-a(n-k_1)^3-(k-n)^3\rangle^{2-2b}}\\ \lesssim \sup_k\langle k\rangle^{2s_1-4s} \sum\limits_{|k_1|>|k|}\langle k_1 \rangle^{-2s}\lesssim \sup_k\langle k \rangle^{2s_1-6s+1}\lesssim 1
\end{multline*}
for $s_1-s\leq 2s-\frac{1}{2}$.
\\{\bf Case B.2.} $|n-k|< \delta|k|$\\
The required estimate specific to this region is $$|n-k_1|\geq (\widetilde{r}_1-1)|k|-|k_1-\widetilde{r}_1k|-|n-k|>(\widetilde{r}_1-1-2\delta)|k|.$$ Also the restriction $|n-k|<\delta|k|$ that yields $|n|\lesssim |k|$ is essential for the summability in the $n$-variable.
Let $\eta_j$ be some constants satisfying $|\eta_j|<\delta$, $j=1,2$, for which $n-k=\eta_1k$ and $k_1-\widetilde{r}_1k=\eta_2k$. Then 
\begin{multline*}
|k^3-ak_1^3-a(n-k_1)^3-(k-n)^3|=|k^3-a(\widetilde{r}_1+\eta_2)^3k^3-a(1+\eta_1-\widetilde{r}_1-\eta_2)^3k^3-\eta_1^3k^3|\\=|\big(1-a+3a\widetilde{r}_1(1-\widetilde{r}_1)+\mathcal{O}(\delta) \big)k^3|\gtrsim |k|^3.
\end{multline*}
Using the above estimates 
\begin{multline*}
    \eqref{8}\lesssim \sup_k\langle k \rangle^{2s_1-2s-6+6b}\underset{\substack{|k_1|>|k|\\ |n|\lesssim |k|}}{\sum}\langle k_1 \rangle^{-2s}\lesssim \sup_k\langle k \rangle^{2s_1-2s-5+6b}\sum\limits_{|k_1|>|k|}\langle k_1 \rangle^{-2s}\\ \lesssim \sup_k\langle k \rangle^{2s_1-4s-4+6b}\lesssim 1,
\end{multline*}
as long as $s_1-s\leq s+2-3b$.
\\{\bf Case B.3.} $|n-k_1|< \delta|k|$\\
Here $|n-k|\approx|k|$, since $(\widetilde{r}_1-1-2\delta)|k|<|n-k|<(\widetilde{r}_1-1+2\delta)|k|$. Hence for $s_1-s\leq s-\frac{1}{2}$, we have
\begin{multline*}
    \eqref{8}\lesssim  \sup_k\langle k \rangle^{2s_1-2s}\underset{\substack{|k_1|>|k|\\ n\neq 0}}{\sum}\frac{\langle k_1 \rangle^{-2s}}{\langle k^3-ak_1^3-a(n-k_1)^3-(k-n)^3\rangle^{2-2b}}\\ \lesssim \sup_k\langle k \rangle^{2s_1-2s}\sum\limits_{|k_1|>|k|}\langle k_1 \rangle^{-2s}\lesssim \sup_k\langle k \rangle^{2s_1-4s+1}\lesssim 1.
\end{multline*}
{\bf Case C.} $|k_1-\widetilde{r}_1k|\geq \delta |k|$, $|k_1-\widetilde{r}_2k|\geq \delta |k|$ \\
Note that $|k_1-k|\lesssim |k_1-\widetilde{r}_1k|$, because $|k_1-k|\leq |k_1-\widetilde{r}_1k|+(\widetilde{r}_1-1)|k|\leq \big(\frac{\widetilde{r}_1-1+\delta}{\delta}\big)|k_1-\widetilde{r}_1k|$. We need to bound
\begin{multline*}
\eqref{7}\lesssim \sup_k\langle k \rangle^{2s_1-2}\underset{\substack{k_1\neq 0\\ n\neq 0}}{\sum}\frac{\langle k_1 \rangle^{-2s-2}\langle n-k_1 \rangle^{-2s+1}\langle n-k\rangle^{-2s+1}}{\langle k^3-ak_1^3-a(n-k_1)^3-(k-n)^3\rangle^{2-2b}}\\+ \sup_k\langle k\rangle^{2s_1-2}\underset{\substack{k_1\neq 0\\ n\neq 0}}{\sum}\frac{\langle k_1 \rangle^{-2s-2}\langle n-k_1 \rangle^{-2s}\langle n-k\rangle^{-2s+1}|k-k_1|}{\langle k^3-ak_1^3-a(n-k_1)^3-(k-n)^3\rangle^{2-2b}}=:\text{I}_1+\text{I}_2.  
\end{multline*}
{\bf Case C.1.} $|k_1|\geq \delta |k|$\\
In this case, since $-2s+1<0$,
\begin{multline*}
    \text{I}_1\lesssim \sup_k\langle k \rangle^{2s_1-2}\underset{\substack{k_1\neq 0\\ n\neq 0}}{\sum}\frac{\langle k_1 \rangle^{-2s-2}\langle k_1-k \rangle^{-2s+1}}{\langle k^3-ak_1^3-a(n-k_1)^3-(k-n)^3\rangle^{2-2b}} \\ \lesssim \sup_k\langle k \rangle^{2s_1-4}\sum\limits_{k_1\neq 0} \langle k_1 \rangle^{-2s}\langle k_1-k \rangle^{-2s+1}\lesssim \sup_k\langle k\rangle^{2s_1-2s-3}\lesssim 1
\end{multline*}
whenever $s_1-s\leq \frac{3}{2}$. In the same way, the boundedness of $\text{I}_2$ can be shown provided that $s_1-s\leq 1$.
\\{\bf Case C.2.} $|k_1|< \delta |k|$\\
In this case, $|k-k_1|\lesssim |k|$ implies the bound
\begin{multline*}
\text{I}_1+ \text{I}_2\lesssim \sup_k\langle k \rangle^{2s_1-2}\underset{\substack{k_1\neq 0\\ n\neq 0}}{\sum}\frac{\langle k_1 \rangle^{-2s-2}\langle n-k_1 \rangle^{-2s+1}\langle n-k\rangle^{-2s+1}}{\langle k^3-ak_1^3-a(n-k_1)^3-(k-n)^3\rangle^{2-2b}}\\+ \sup_k\langle k\rangle^{2s_1-1}\underset{\substack{k_1\neq 0\\ n\neq 0}}{\sum}\frac{\langle k_1 \rangle^{-2s-2}\langle n-k_1 \rangle^{-2s}\langle n-k\rangle^{-2s+1}}{\langle k^3-ak_1^3-a(n-k_1)^3-(k-n)^3\rangle^{2-2b}}=:\text{J}_1+\text{J}_2.   \end{multline*}
\\{\bf Case C.2.1.} $|n-k_1|\leq \delta |k|$\\
 Note here that $|n-k|\geq |k|-|k_1|-|n-k_1|\geq (1-2\delta)|k|$. Moreover, since $u$ is mean zero, we have, for some $\eta_j\neq0$ satisfying $|k|^{-1}\leq|\eta_j|<\delta$, $j=1,2$, that $k_1=\eta_1k$ and $n-k_1=\eta_2k$. Hence the restriction $(\eta_1+\eta_2)k=n\neq 0$ provides us with a parameter $\eta:=\eta_1+\eta_2$ satisfying $|k|^{-1}\leq |\eta|<2\delta$, and yielding the following bound
 \begin{multline*}
 |k^3-ak_1^3-a(n-k_1)^3-(k-n)^3|=|\big(1-a(\eta_1^3+\eta_2^3)-(1-\eta)^3\big)k^3|\\=|k|^3|\eta||3(1-\eta)+\eta^2(1-a)+3a\eta_1\eta_2| \gtrsim |k|^2 . 
 \end{multline*}
 Exploiting the above estimates we arrive at 
\begin{align*}
 \text{J}_1 \lesssim \sup_k\langle k \rangle^{2s_1-2s-5+4b} \underset{\substack{k_1\neq 0\\ |n|\lesssim |k|}}{\sum} \langle k_1 \rangle^{-2s-2}\lesssim \sup_k\langle k \rangle^{2s_1-2s-4+4b}\lesssim 1  
\end{align*}
and
\begin{align*}
 \text{J}_2 &\lesssim \sup_k\langle k \rangle^{2s_1-5+4b} \underset{\substack{k_1\neq 0\\ n\neq 0}}\sum \langle k_1 \rangle^{-2s-2}\langle n-k_1 \rangle^{-2s}\langle n-k \rangle^{-2s+1} \\& \lesssim \sup_k\langle k \rangle^{2s_1-5+4b}\sum\limits_{k_1\neq0}\langle k_1 \rangle^{-2s-2}\langle k_1-k \rangle^{-2s+1}\lesssim \sup_k\langle k \rangle^{2s_1-2s-4+4b}\lesssim 1  
\end{align*}
provided that $s_1-s\leq 2-2b$.
\\{\bf Case C.2.2.} $|n-k|\leq \delta |k|$\\
In this region, $|n-k_1|\geq |k|-|k_1|-|n-k|>(1-2\delta)|k|$. Then, we write $k_1=\eta_1k$ for some $\eta_1$ with $|k|^{-1}\leq|\eta_1|<\delta$, and $n-k=\eta_2k$ for some $\eta_2$ with $0\leq |\eta_2|\leq\delta$. Via these
\begin{multline*}
|k^3-ak_1^3-a(n-k_1)^3-(k-n)^3|=|\big(1+\eta_2^3-a(\eta_1^3+(1-\eta_1+\eta_2)^3)\big)k^3|=|k|^3|1-a+\mathcal{O}(\delta)|\\\gtrsim |k|^3.  
\end{multline*}
As above we get
\begin{align*}
 \text{J}_1+\text{J}_2\lesssim \sup_k\langle k \rangle^{2s_1-2s-6+6b}\lesssim 1  
\end{align*}
whenever $s_1-s\leq 3-3b$.

\section{Existence of Global Attractor}
This section is devoted to the proof of Theorem \ref{attrhir}. We consider the system
\begin{align}\label{disshir}
\begin{cases}
u_t+  au_{xxx}+\gamma u+3a (u^2)_x+\beta(v^2)_x =f \\ v_t+v_{xxx}+\gamma v+3uv_x=g \\ (u,v)\rvert_{t = 0}=(u_0,v_0)\in \dot{H}^1(\mathbb{T})\times H^1(\mathbb{T}).
\end{cases}
\end{align}
Recall that $\beta<0$. Firstly we show the existence of an absorbing set corresponding to the system \eqref{disshir}. To achieve this we use conserved energies \eqref{conservationhir} to obtain:  
\begin{lemma}\label{energyhir}
Let $(u,v)$ be a solution of the system \eqref{disshir} with data $(u_0,v_0)$, we have the a priori estimate:
\begin{align*}
    \norm{u(t)}_{H^1}+\norm{v(t)}_{H^1}\leq C=C(a,\beta,\gamma,\norm{u_0}_{H^1},\norm{v_0}_{H^1},\norm{f}_{H^1},\norm{g}_{H^1}),
\end{align*}
for $t>0$.
\end{lemma}
\begin{proof}
We start by noting that the constants in the following calculations are denoted by $C$, $C_0$, and $C_1$ whose value may change, their dependence are to be highlighted though. To obtain the $L^2$ bounds for $u$ and $v$, we use $E_1(t):=E_1(u,v)(t)=\norm{u}_{L^2}^2-\frac{2\beta}{3}\norm{v}_{L^2}^2$. Thus
\begin{align*}
\partial_tE_1(t)+2\gamma E_1(t)=2\int uf-\frac{2\beta}{3}vg\,\text{d}x &\leq 2\norm{u}_{L^2}\norm{f}_{L^2}-\frac{4\beta}{3}\norm{v}_{L^2} \norm{g}_{L^2}\\&\leq 2(\sqrt{2}\norm{f}_{L^2}+\sqrt{-\beta}\norm{g}_{L^2})\sqrt{E_1(t)}.  
\end{align*}
Setting $E_1(t)=e^{-2\gamma t}F_1(t)$ and using the above inequality we obtain
\begin{align*}
\partial_t\sqrt{F_1(t)}\leq e^{\gamma t}(\sqrt{2}\norm{f}_{L^2}+\sqrt{-\beta}\norm{g}_{L^2}).    
\end{align*}
Integrating this inequality from $0$
to $t$ and then utilizing the resulting inequality in the norms of $u$ and $v$, we arrive at  
\begin{multline*}
\norm{u(t)}_{L^2}+\sqrt{\frac{-2\beta}{3}}\norm{v(t)}_{L^2}\\ \leq \sqrt{2}e^{-\gamma t}\sqrt{\norm{u_0}_{L^2}^2-\frac{2\beta}{3}\norm{v_0}_{L^2}^2}+\frac{1-e^{-\gamma t}}{\gamma} (2\norm{f}_{L^2}+\sqrt{-2\beta}\norm{g}_{L^2}). 
\end{multline*}
Regarding the bounds for the spatial derivatives of $u$ and $v$, we consider $E_2(t):=E_2(u,v)(t)=(1-a)\big(\norm{u_x}_{L^2}^2-2\int u^3\text{d}x\big)-2\beta\big(\norm{v_x}_{L^2}^2-\int uv^2\text{d}x\big)$. Note that
\begin{multline*}
(1-a)\norm{u_x}_{L^2}^2-2\beta\norm{v_x}_{L^2}^2=E_2(t)+2(1-a)\int u^3\,\text{d}x-2\beta\int uv^2\,\text{d}x \\ \leq E_2(t)+C\norm{u}_{H^1}\big(\norm{u}_{L^2}^2+\norm{v}_{L^2}^2\big)\leq E_2(t)+C+C\norm{u_x}_{L^2}    
\end{multline*}
the constants depend on the bounds on $\norm{u}_{L^2}$, $\norm{v}_{L^2}$ in the final inequality. By this inequality, we have
\begin{multline*}
\sqrt{1-a}\norm{u_x}_{L^2}-\frac{C}{2\sqrt{1-a}}\leq \sqrt{\Big(\sqrt{1-a}\norm{u_x}_{L^2}-\frac{C}{2\sqrt{1-a}}\Big)^2-2\beta\norm{v_x}_{L^2}^2}\\ \leq \sqrt{E_2(t)+C+C^2/4(1-a)}\lesssim \sqrt{|E_2(t)|}+C   
\end{multline*}
and \begin{align*}
\sqrt{-2b}\norm{v_x}_{L^2}\leq \sqrt{\Big(\sqrt{1-a}\norm{u_x}_{L^2}-\frac{C}{2\sqrt{1-a}}\Big)^2-2\beta\norm{v_x}_{L^2}^2}\lesssim \sqrt{|E_2(t)|}+C.  
\end{align*} Thus $\norm{u_x(t)}_{L^2}+\norm{v_x(t)}_{L^2}\lesssim \sqrt{|E_2(t)|}+C$. To end up the argument it suffices to show that $E_2$ is bounded.
Using this bound and the embedding $H^1\hookrightarrow L^{\infty}$ we obtain
\begin{multline*}
\partial_tE_2(t)+2\gamma E_2(t)=2(1-a)\int f_xu_x-3fu^2+\gamma u^3 \,\text{d}x-2\beta\int 2g_xv_x-fv^2-2guv+\gamma uv^2\,\text{d}x \\ \leq C_0(\norm{u_x}_{L^2}+\norm{v_x}_{L^2})+C_1\leq C_0\sqrt{|E_2(t)|}+C_1 
\end{multline*}
where the constants $C_0$, $C_1$ depend on the norms $\norm{f}_{H^1}$, $\norm{g}_{H^1}$, the constants $a$, $\beta$, $\gamma$ and the bounds on $\norm{u}_{L^2}$, $\norm{v}_{L^2}$. Setting $E_2(t)=e^{-2\gamma t}F_2(t)$, we get that
\begin{align*}
\partial_tF_2(t)\leq e^{\gamma t}\big(C_0\sqrt{|F_2(t)|}+C_1e^{\gamma t}\big),    
\end{align*}
from which we have that
\begin{align*}
E_2(t)&\leq e^{-2\delta t}E_2(0)+C_1\frac{1-e^{-2\gamma t}}{2\gamma}+C_0\int_0^te^{-2\gamma (t-t')}\sqrt{|E_2(t')|}\,\text{d}t'\\& \leq |E_2(0)|+C_1+C_0\norm{\sqrt{|E_2|}}_{L^{\infty}([0,t])}  
\end{align*}
for $t>0$. This shows that $E_2$ is bounded from above because if it  were the case that $t$ might be the first time at which $E_2$ assumes its largest value, say $C$, over $[0,t]$ with $E_2(t)=C\gg |E_2(0)|+C_1+C_0=:\widetilde{C}$, then by the above inequality we would have $C\leq \widetilde{C} (1+\sqrt{C})$, but this is impossible for sufficiently large $C\gg 1$. Also the Sobolev embedding and the bounds on $\norm{u}_{L^2}$, $\norm{v}_{L^2}$ suggest that $E_2$ is bounded below. 
\end{proof}
As a consequence of the Lemma \ref{energyhir}, the existence of an absorbing ball $\mathcal{B}_0\subset H^1\times H^1$ follows. As for the verification of the asymptotic compactness of the flow, the second task is to obtain smoothing estimate as done in the non-dissipative case. 
\begin{theorem}\label{smoothingdisshir}
 Consider the solution of \eqref{disshir} with initial data $(u_0,v_0)\in \dot{H}^1\times H^1$. Then for any $\alpha<\min\{\frac{1}{2}, 3-\mu(\rho_a)\}$, we have
 \begin{multline*}
    \norm{u(t)-e^{-(a\partial_x^3+\gamma)t}u_0-\int_0^te^{-(a\partial_x^3+\gamma)(t-r)}\rho_2(v,v)(r)\,\text{d}r}_{H^{1+\alpha}}\\+ \norm{v(t)-e^{-(\partial_x^3+\gamma)t}v_0-\int_0^te^{-(\partial_x^3+\gamma)(t-r)}\rho_3(u,v)(r)\,\text{d}r}_{H^{1+\alpha}}\\ \leq C(\alpha, \gamma, \norm{u_0}_{H^1}, \norm{v_0}_{H^1},\norm{f}_{H^1},\norm{g}_{H^1})
 \end{multline*}
 where $\rho_2$ and $\rho_3$ are as in Proposition \ref{base}.
\end{theorem}
\begin{proof}
We write the system \eqref{disshir} by the Fourier transform as follows
\begin{align}\label{fourierdisshir}
\begin{cases}
\partial_t u_k-(ia k^3-\gamma)u_k+3ia k  \sum\limits_{k_1+k_2=k}u_{k_1}u_{k_2}  +i\beta k\sum\limits_{k_1+k_2=k}v_{k_1}v_{k_2}=f_k \\ \partial_t v_k-(ik^3-\gamma)v_k+3i \sum\limits_{k_1+k_2=k}k_2u_{k_1}v_{k_2}=g_k.
\end{cases}
\end{align}
Using the change of variables $y_k=e^{-ia k^3t+\gamma t}u_k$, $z_k=e^{-ik^3t+\gamma t}v_k$, and $d_k=e^{-ia k^3t+\gamma t}f_k$, $h_k=e^{-ik^3t+\gamma t}g_k$, the above system transforms to
\begin{align*}
\begin{cases}
\partial_ty_k=-3iak \sum\limits_{k_1+k_2=k}e^{-ia t(k^3-k_1^3-k_2^3)}y_{k_1}y_{k_2}-i\beta k \sum\limits_{k_1+k_2=k}e^{-it(a k^3-k_1^3-k_2^3)}z_{k_1}z_{k_2}+d_k \\ \partial_tz_k=-3i\sum\limits_{k_1+k_2=k}k_2e^{-it(k^3-a k_1^3-k_2^3)}y_{k_1}z_{k_2}+h_k.
\end{cases}
\end{align*}
After differentiation by parts as in Proposition \ref{base}, the system \eqref{fourierdisshir} can be written in the form
\begin{multline*}
\begin{cases} 
\partial_t\Big[e^{-ia k^3t+\gamma t}u_k\Big]+e^{-\gamma t}\partial_t\Big[e^{-ia k^3t+2\gamma t}(B_1(u,u)_k+B_2(v,v)_k)\Big]=\\\hspace{3cm}e^{-ia k^3t+\gamma t}\big[R_1(u,v,v)_k+R_2(u,u,u)_k+R_3(u,v,v)_k+2B_1(u,f)_k\\\hspace{8cm}+2B_2(g,v)_k+\rho_1(u,u)_k+\rho_2(v,v)_k+f_k\big] \\\partial_t\Big[e^{-ik^3t+\gamma t}v_k\Big]+e^{-\gamma t}\partial_t\Big[e^{-ik^3t+2\gamma t}B_3(u,v)_k\Big]=\\\hspace{5cm}e^{-ik^3t+\gamma t}\big[R_4(u,u,v)_k+\frac{\beta}{3a}R_4(v,v,v)_k+R_5(u,u,v)_k\\\hspace{8cm}+B_3(f,v)_k+B_3(u,g)_k+\rho_3(u,v)_k+g_k\big], 
\end{cases} 
\end{multline*} 
where $B_j$, $R_j$, and $\rho_j$ are as in Proposition \ref{base}. Integrating these equations from $0$ to $t$ leads to the equations
\begin{multline*} 
u_k(t)-e^{ia k^3t-\gamma t}u_k(0)=-B_1(u,u)_k-B_2(v,v)_k+e^{ia k^3t-\gamma t}\big[B_1(u_0,u_0)_k+B_2(v_0,v_0)_k\big]\\+\int_0^te^{(ia k^3-\gamma) (t-s)}\big[-\gamma B_1(u,u)_k-\gamma B_2(v,v)_k+\rho_1(u,u)_k+\rho_2(v,v)_k+f_k+2B_1(u,f)_k\\\hspace{5cm}+2B_2(g,v)_k+R_1(u,v,v)_k+R_2(u,u,u)_k+R_3(u,v,v)_k\big]\text{d}s\\  v_k(t)-e^{ik^3t-\gamma t}v_k(0)=-B_3(u,v)_k+e^{ik^3t-\gamma t}B_3(u_0,v_0)_k+\int_0^te^{(ik^3-\gamma)(t-s)}\big[-\gamma B_3(u,v)_k\\+\rho_3(u,v)_k+g_k+R_4(u,u,v)_k +\frac{\beta}{3a}R_4(v,v,v)_k+R_5(u,u,v)_k+B_3(f,v)_k+B_3(u,g)\big]ds.
\end{multline*}
Note that
\begin{align*}
\norm{\int_0^te^{(-a\partial_x^3-\gamma)(t-s)}f(x)\text{d}s}_{H^{1+\alpha}}=\norm{\frac{\langle k\rangle^{1+\alpha}f_k}{iak^3-\gamma}(1-e^{(iak^3-\gamma)t})} _{\ell^2_k}\lesssim \norm{f}_{H^{\alpha-2}},   
\end{align*}
analogous estimate holds for $e^{(-\partial_x^3-\gamma)(t-s)}g$ as well. These bounds, the estimates utilized in obtaining main smoothing result, and the growth bound of Lemma \ref{energyhir} yield, for $t<\delta$, that
 \begin{multline*}
    \norm{u(t)-e^{-(a\partial_x^3+\gamma)t}u_0-\int_0^te^{-(a\partial_x^3+\gamma)(t-r)}\rho_2(v,v)(r)\,\text{d}r}_{H^{1+\alpha}}\\+ \norm{v(t)-e^{-(\partial_x^3+\gamma)t}v_0-\int_0^te^{-(\partial_x^3+\gamma)(t-r)}\rho_3(u,v)(r)\,\text{d}r}_{H^{1+\alpha}}\\ \lesssim \norm{f}_{H^{\alpha-2}}+\norm{g}_{H^{\alpha-2}}+\big(\norm{f}_{H^1}+\norm{g}_{H^1}+\norm{u_0}_{H^1}+\norm{v_0}_{H^1}\big)^2+\big(\norm{u}_{X_{a,\delta}^{1,1/2}}+\norm{v}_{X_{1,\delta}^{1,1/2}}\big)^3\\ \leq C\big(\alpha, \gamma, \norm{f}_{H^1}, \norm{g}_{H^1}, \norm{u_0}_{H^1}, \norm{v_0}_{H^1} \big)
 \end{multline*}
where we use the local theory bounds for $X_{a,\delta}^{1,1/2}$, $X_{1,\delta}^{1,1/2}$ norms for the local existence time $\delta$ in the final inequality. By virtue of dissipation, this bound also holds for arbitrarily large times making use of the local bound above, for the full discussion, see Section $6$ in \cite{erdoganzak}.    
\end{proof}

\begin{proof}[Proof of Theorem \ref{attrhir}]
For the existence of a global attractor, we check the asymptotic compactness of the flow. It suffices to show that for any sequence $(u_{0,r},v_{0,r})$ in an absorbing set $\mathcal{B}_0$ and for any sequence of times $t_r\rightarrow \infty$, the sequence $U_{t_r}(u_{0,r},v_{0,r})$ possesses a convergent subsequence in $\dot{H}^1\times H^1$. Next we use Theorem \ref{smoothingdisshir}, for almost every $a\in(\frac{1}{4},1)$ such that $\alpha<\frac{1}{2}$ and $\rho_2=\rho_3=0$, to write
\begin{align*}
  U_{t_r}(u_{0,r},v_{0,r})=(e^{-(a\partial_x^3+\gamma)t_r}u_{0,r},e^{-(\partial_x^3+\gamma)t_r}v_{0,r})+N_{t_r}(u_{0,r},v_{0,r})   
\end{align*}
where the nonlinear part $N_{t_r}(u_{0,r},v_{0,r})$ is contained within a ball in $H^{1+\alpha}\times H^{1+\alpha}$. Therefore by Rellich's theorem the sequence $\{N_{t_r}(u_{0,r},v_{0,r}):r\in\mathbb{N}\}$ has a convergent subsequence in $H^1\times H^1$. This implies the existence of a convergent subsequence of the sequence $\{U_{t_r}(u_{0,r},v_{0,r}):r\in\mathbb{N}\}$, since  
\begin{align*}
\norm{(e^{-(a\partial_x^3+\gamma)t_r}u_{0,r},e^{-(\partial_x^3+\gamma)t_r}v_{0,r})}_{H^1\times H^1}\lesssim e^{-\gamma t_r}(\norm{u_{0,r}}_{H^1},\norm{v_{0,r}}_{H^1}) \lesssim e^{-\gamma t_r}\rightarrow 0  
\end{align*}
as $t_r\rightarrow \infty$ uniformly. Therefore $U_t$ is asymptotically compact. To prove the compactness of the attractor $\mathcal{A}$ in the space $H^{1+\alpha}\times H^{1+\alpha}$ for any $\alpha\in (0,\frac{1}{2})$, we need to show, by using Rellich's theorem, that the attractor is bounded in $H^{1+\alpha+\epsilon}\times H^{1+\alpha+\epsilon}$ for some $\epsilon>0$ satisfying $\alpha+\epsilon<\frac{1}{2}$. In this regard, it suffices to find some closed ball $\mathcal{B}_{\alpha+\epsilon}\subset H^{1+\alpha+\epsilon}\times H^{1+\alpha+\epsilon}$ such that $\mathcal{A}\subset \mathcal{B}_{\alpha+\epsilon}$ where
\begin{align*}
    \mathcal{A}=\bigcap_{\tau\geq 0}\overline{\bigcup_{t\geq \tau}U_t\mathcal{B}_0}=:\bigcap_{\tau\geq 0}V_{\tau}.
\end{align*}
As above, using Theorem \ref{smoothingdisshir}, we can express each element of $V_{\tau}$ as a sum of linear evolution which decays to zero exponentially and the nonlinear evolution contained by some ball $\mathcal{B}_{\alpha+\epsilon}$ in $H^{1+\alpha+\epsilon}\times H^{1+\alpha+\epsilon}$. This implies that the set $V_{\tau}$ is contained in a $\delta_{\tau}$ neighbourhood $N_{\tau}$ of $\mathcal{B}_{\alpha+\epsilon}$ in $H^{1+\alpha+\epsilon}\times H^{1+\alpha+\epsilon}$. Here $\delta_{\tau}\rightarrow 0$ as $\tau \rightarrow \infty$ due to the exponential decay of linear evolutions. Therefore,
\begin{align*}
    \mathcal{A}=\bigcap_{\tau\geq 0}V_{\tau}\subset \bigcap_{\tau\geq 0}N_{\tau}=\mathcal{B}_{\alpha+\epsilon}.
\end{align*}

\end{proof}
\section{APPENDIX} \label{sec:appendix}
The following lemma is used repeatedly in the text. For a proof see for instance \cite{erdoganzak}.
\begin{lemma} \label{cal.lem}
\begin{enumerate}
    \item If $\beta\geq \gamma\geq 0$ and $\gamma+\beta>1$,
    $$\sum_n\frac{1}{\langle n-k_1\rangle^{\beta}\langle n-k_2\rangle^{\gamma}}\lesssim \langle k_1-k_2 \rangle^{-\gamma}\varphi_{\beta}(k_1-k_2)$$
where \begin{align*}
    \varphi_{\beta}(k)=\begin{cases}
    1, \hspace{2.4cm} \text{if}\,\, \beta>1\\ \log(1+\langle k\rangle), \hspace{0.5cm}\text{if}\,\, \beta=1\\ \langle k\rangle^{1-\beta}, \hspace{1.5cm}\text{if}\,\, \beta<1.
    \end{cases}
\end{align*}
\item If $\beta>\frac{1}{2}$ and $\gamma>\frac{1}{3}$, then we have
\begin{align*}
     \sum_n\frac{1}{\langle n^2+an+b \rangle^{\beta}}\lesssim 1,\,\,\text{and}\hspace{0.5cm} \sum_n\frac{1}{\langle n^3+an^2+bn+c \rangle^{\gamma}}\lesssim 1 
\end{align*}
    where the implicit constants are independent of $a,b$ and $c$.
\end{enumerate}

\end{lemma}

\end{document}